\documentclass[a4paper, 12pt]{amsart}
\usepackage{amssymb}
\usepackage{amsmath}
\usepackage{amsthm}
\usepackage{amscd}
\usepackage{comment}
\usepackage[all]{xy}
\usepackage{color}

\title{Lipschitz homotopy convergence of Alexandrov spaces}

\author{Ayato Mitsuishi}
\email[A.~Mitsuishi]{{mitsuishi@fukuoka-u.ac.jp}}
\address{{Department of Applied Mathematics, Fukuoka University, Jyonan-ku, Fukuoka-shi, Fukuoka 814-0180, JAPAN}}

\author{Takao Yamaguchi}
\email[T.~Yamaguchi]{takaoy@math.kyoto-u.ac.jp}
\address
{Department of Mathematics, Kyoto University, Kitashirakawa, Kyoto 606-8502, JAPAN}

\date{\today}

\theoremstyle{plain}
\newtheorem{thm}{Theorem}[section]
\newtheorem{lem}[thm]{Lemma}
\newtheorem{sublem}[thm]{Sublemma}
\newtheorem{cor}[thm]{Corollary}
\newtheorem{prop}[thm]{Proposition}
\newtheorem{defn}[thm]{Definition}
\newtheorem{rem}[thm]{Remark}

\newtheorem{claim}[thm]{Claim}

\makeatletter

\newcommand{\diam}[0]{\mathrm{diam}\,}
\newcommand{\e}[0]{\epsilon}

\newcommand{\pa}[0]{\partial}
\newcommand{\ca}[0]{\mathcal}

\newcommand{\pmed}[0]{\par\medskip}
\newcommand{\psmall}[0]{\par\smallskip}

\newcommand{\n}[0]{\noindent}

\newcommand{\beq}[0]{\begin{equation}}
\newcommand{\eeq}[0]{\end{equation}}

\newcommand{\benum}[0]{\begin{enumerate}}
\newcommand{\eenum}[0]{\end{enumerate}}

\newcommand{\supp}[0]{\mathrm{supp}}

\begin{document}

\begin{abstract}
%In the present paper, we prove the Lipschitz homotopy convergence in 
%the moduli space of $n$-dimensional compact Alexandrov spaces with 
%curvature $\ge -1$, diameter $\le D$ and volume $\ge v_0$ for given $n$, $D,v_0>0$,
%which yields the finiteness of Lipschitz homotopy types of spaces in the moduli space.
We introduce the notion of good coverings of metric spaces, and
prove that if a metric space admits a good covering, then it has the same locally Lipschitz homotopy type as the nerve complex of the covering. 
As an application, we obtain a Lipschitz homotopy stability result for a moduli space of compact Alexandrov spaces without collapsing. 
%that a moduli space of Alexandrov spaces with fixed lower curvature bound, lower volume bound, and upper diameter bound, contains finitely many Lipschitz homotopy types.
\end{abstract}
\maketitle
%\tableofcontents

%%%%%%%%%start section {introduction}

\section{Introduction}
For given $n$ and $D, v_0>0$, 
let $\mathcal A(n,D,v_0)$ denote the set of isometry classes of compact $n$-dimensional 
Alexandrov spaces with curvature $\ge -1$, diameter $\le D$ and volume $\ge v_0$.
Perelman's stability theorem has played important roles in the geometry of 
Alexandrov spaces with curvature bounded below. 
This theorem implies that the set of homeomorphism classes of spaces in $\mathcal A(n,D,v_0)$ is finite. 
Although he also claimed the Lipschitz version of the stability theorem is true, it has not yet been appeared.

We formulate our results for general metric spaces having good coverings.
We say that a locally finite open covering of a metric space is {\it good} if any non-empty
intersection in the covering has a Lipschitz strong deformation
retraction to a point %in it 
(see Definition \ref{def:good cover} for the detail).

We use a symbol $\tau(\e_1, \e_2,\dots, \e_k)$ to denote a positive continuous function
%stant %depending on positive numbers $\epsilon_1, \epsilon_2, \dots, \epsilon_k$ 
satisfying 
$\lim_{\e_1, \e_2, \dots, \e_k \to 0} \tau(\e_1, \e_2, \dots, \e_k) = 0$. 
%\[
%\lim_{\e_1, \e_2 \dots, \e_k \to 0} \tau(\e_1, \e_2, \dots, \e_k) = 0. 
%\]

The main theorems of the present paper are stated as follows.

\begin{thm} \label{thm:lip-nerve}
Let $M$ be a $\sigma$-compact metric space having a 
good covering $\ca U$. Then $M$ has the same locally Lipschitz homotopy type
as the nerve of $\ca U$. 
\end{thm}

We remark that in Theorem \ref{thm:lip-nerve} if $M$ is compact, 
it has the same Lipschitz homotopy type as the nerve of $\ca U$.

\begin{thm} \label{thm:lip-conv}
There exists a positive number $\e=\e_n(D,v_0)$ such that 
if $M,M'\in \mathcal A(n,D,v_0)$ have the Gromov-Hausdorff distance $d_{GH}(M,M') <\e$,
then $M$ has the same Lipschitz homotopy type as $M'$.
More precisely if $\theta:M\to M'$ is an $\e$-approximation, then 
there is a Lipschitz homotopy equivalence $f:M\to M'$ satisfying
that $|f(x),\theta(x)|<\tau(\e)$ for all $x\in M$.
\end{thm}

%%Here and hereafter 
%{From here on,}
%we use the symbol $\tau(\e)$ to denote a positive 
%function depending on $\e >0$ with $\lim_{\e\to 0}\tau(\e) =0$.

%Since any finite dimensional Alexandrov space admits a good covering (see \cite{MY:good}),
As a direct consequence of Theorem \ref{thm:lip-conv}, we have 

\begin{cor} \label{thm:lip-homo-type}
The set of Lipschitz homotopy types of Alexandrov spaces in 
$\mathcal A(n,D,v_0)$ is finite.
\end{cor}

This provides a weaker version of ``the finiteness of bi-Lipschitz homeomorphism classes'' mentioned above. 
%%%claimed by Perelman. 

In Corollary	 \ref{thm:lip-homo-type} we prove that every $M$ and $M'$ in $\mathcal A(n,D,v_0)$ with 
small Gromov-Hausdorff distance have the same Lipschitz homotopy type through isomorphic 
nerves of some good coverings on them. 
However it was shown in \cite{BGP} and \cite{Ya:conv} that there is an almost isometric map
from a closed domain of an almost regular part of $M$ to a closed domain of an almost regular part of $M'$.
John Lott asked us if one can extend such an almost isometric map to a Lipschitz homotopy equivalence $M\to M'$. 
The answer is yes:

\begin{thm} \label{thm:gluing}
Let $\delta$ be a sufficiently small positive number with respect to $n$.
For given compact $n$-dimensional Alexandrov space $M$ with curvature $\ge -1$ and 
a closed domain $D$ in the $\delta$-regular part of $M$, there exists 
an $\e=\e_{M,D}>0$ satisfying the following: Let $M'$ be a compact $n$-dimensional Alexandrov space
with curvature $\ge -1$ and with $d_{GH}(M,M')<\e$, and let $\theta:M\to M'$ be an $\e$-approximation.
Then there is a Lipschitz homotopy equivalence $f:M\to M'$ such that
\begin{enumerate}
\item the restriction of $f$ to $D$ is $\tau(\e)$-almost isometric;
\item $|f(x),\theta(x)|<\tau(\e)$ for all $x\in M$.
\end{enumerate}
\end{thm}

Theorem \ref{thm:lip-nerve} has an application to the set of homotopies of mapping between two metric spaces.
Let $[X,Y]$ denote the set of all homotopy classes of continuous maps from $X$ to $Y$, 
and $[X,Y]_{\mathrm{loc}\text{-}\mathrm{Lip}}$ the set of all locally Lipschitz homotopy classes of locally Lipschitz maps from $X$ to $Y$.
In Corollary 1.3 of \cite{MY:LLC}, we proved that if $K$ is a simplicial complex and $Y$ is a locally Lipschitz contractible metric space,
then the natural map $[K,Y]_{\mathrm{loc}\text{-}\mathrm{Lip}} \to [K, Y]$ is bijective. 

Using Theorem \ref{thm:lip-nerve} and Corollary 1.3 of \cite{MY:LLC}, we obtain the following. 

%We denote by $\mathrm{Lip}_{\mathrm{loc}}(X,Y)$ the set of all locally Lipschitz maps between metric spaces $X$ and $Y$.

\begin{cor} \label{cor:Lip-conti}
Let $X$ be a $\sigma$-compact metric space admitting a good covering, 
and $Y$ a locally Lipschitz contractible metric space. 
Then, the natural map $[X,Y]_{\mathrm{loc}\text{-}\mathrm{Lip}} \to [X,Y]$ is bijective.

In particular, every continuous map from $X$ to $Y$ is homotopic to a locally Lipschitz one.
\end{cor}

As an immediate consequence of Corollary \ref{cor:Lip-conti}, we have 
the following for instance.

\begin{cor} \label{cor:Lip-alex}
Let $X$ be a finite-dimensional compact Alexandrov space with curvature bounded below,
and $Y$ a locally Lipschitz contractible metric space. 
Then every continuous map from $X$ to $Y$ is homotopic to a Lipschitz map.
\end{cor}

\vspace{1em}
\noindent
{\bf Organization}.
The rest of the present paper consists of Sections \ref{sec:prelim}--\ref{sec:gluing}. 
In Section \ref{sec:prelim}, we recall the notions of Lipschitz homotopies, Alexandrov spaces and good coverings needed in this paper.
Sections \ref{sec:Lip-homotopy} and \ref{sec:non-cpt} are devoted to prove Theorem \ref{thm:lip-nerve}, where 
we employ a basic strategy in the proof of Theorem 9.4.15 of \cite{Sha}. 
Since the argument in \cite{Sha} is only topological, we need to proceed in the category of (locally) Lipschitz maps.
%
%
%we develop a general argument based on { Theorem 9.4.15 of \cite{Sha}}.
In Section \ref{sec:Lip-homotopy}, we consider the case when metric spaces are compact, and deal with 
the non-compact case in Section \ref{sec:non-cpt}. Using Theorem \ref{thm:lip-nerve}
and a stability result of nerves of good coverings in \cite{MY:good}, 
we prove Theorem \ref{thm:lip-conv} and Corollaries \ref{thm:lip-homo-type} and \ref{cor:Lip-conti}
in Section \ref{sec:appl}. 
In Section \ref{sec:gluing}, we prove Theorem \ref{thm:gluing} by developing a gluing method in 
\cite{BGP}.

\vspace{1em}\noindent
{\bf Acknowledgement}
The authors would like to thank John Lott for his valuable comment concerning Theorem\ref{thm:gluing}. 
This material is based upon work supported by the National Science Foundation under Grant No. DMS-1440140 
while the second named author was in residence at the Mathematical Sciences Research Institute in Berkeley, 
California, during the Spring 2016 semester.
This work was also supported by JSPS KAKENHI Grant Numbers 26287010, 15H05739, 15K17529.

%%%%%%%%%%%%%%%%%%%%%%%%%%%%%%%%%%%%%%%%%%%%%%%%%%%%%%%%%%%%%%%%%%%%%%%%

\section{Preliminaries} \label{sec:prelim}
In this paper, the distance between two points $x, y$ in a metric space is denoted by $|xy|$ or $|x,y|$. 
The open metric ball around $x$ of radius $r$ is denoted by $B(x,r)$.
To prove the main result, we prepare several terminologies. 
%which are usual phrases replaced ``continuous'' with ``(locally) Lipschitz''
%Roughly speaking, 

\subsection{Homotopies in the category of (locally) Lipschitz maps}
\mbox{}
\pmed

Let $X$ and $Y$ be metric spaces.

\begin{defn} \upshape \label{def:LSDR}
%Let $X$ be and $A$}.
We say that a subset $A$ of $X$ is a {\it locally Lipschitz strong deformation retract} of $X$ if there is a Lipschitz map $F : X \times [0,1] \to X$ such that $F(x,0)=x$, $F(x,1) \in A$ and $F(a,t)=a$ for any $x \in X$, $a \in A$ and $t \in [0,1]$. 
%%%$Y$ is a strong deformation retract of $X$ in the usual sense and a strong deformation retraction can be chosen to be a Lipschitz one.
%%%Such a strong deformation retraction is called a {\it Lipschitz strong deformation retraction}. 
Then, the map $F$ is called a locally Lipschitz strong deformation retraction of $X$ to $A$.
\end{defn}

%Note that a single point is a strong deformation retract of a Euclidean space in the usual sense, but it is not a Lipschitz strong deformation retract of the Euclidean space. 

\begin{defn} \upshape \label{def:Lip-hom}
Two maps $h_0, h_1:X\to Y$ are said to be {\it locally Lipschitz homotopic}
if there exists a locally Lipschitz map $h : X \times [0,1] \to Y$
such that $h_i = h(\cdot, i)$ \, $(i=0,1)$.
%two maps $h_0$ and $h_1$ are said to be {\it {locally} Lipschitz homotopic}.
%In this case, we write as $h_0 \sim_\mathrm{Lip} h_1$.

We say that $X$ and $Y$ are {\it locally Lipschitz homotopy equivalent} if there are locally Lipschitz maps $f : X \to Y$ and $g : Y \to X$ such that 
$g \circ f$ and $f \circ g$ are locally Lipschitz homotopic to $1_X$ and $1_Y$
respectively.
In this case, $f$ and $g$ are called {\it locally Lipschitz homotopy equivalences}.
\end{defn}

In the above definition, if a locally Lipschitz homotopy can be chosen to be a Lipschitz one, then it is called a Lipschitz homotopy. 
For other notions appeared in Definitions \ref{def:LSDR} and \ref{def:Lip-hom}, we use similar terminologies.
From definition, if $Y$ is a (locally) Lipschitz strong deformation retract of $X$, then $X$ and $Y$ are (locally) Lipschitz homotopy equivalent. 

%{Note that} a bounded space can not be a Lipschitz strong deformation retract of an unbounded space. 
%{In particular, if a metric space admits a Lipschitz strong deformation retract to a point, then it is bounded}. 
Let $X$ be an unbounded metric space, and $f:X\to X$ a Lipschitz map
whose image is a bounded subset. Then it follows from definition that
$f$ is not Lipschitz homotopic to $1_X$. 
In particular if a metric space $X$ is Lipschitz homotopy equivalent to a bounded metric space, then $X$ is also bounded.

\subsection{The Gromov-Hausdorff distance}
A map $f : X \to Y$ between metric spaces is called an $\epsilon$-{\it approximation} if it satisfies 
\begin{itemize}
\item $||f(x),f(y)| - |x,y|| < \epsilon$ for all $x,y \in X$; 
\item for any $y \in Y$, there is an $x \in X$ such that $|f(x),y|< \epsilon$.
\end{itemize}

The {\it Gromov-Hausdorff distance $d_{\mathrm{GH}}(X,Y)$} 
between $X$ and $Y$ is defined as 
\[
d_{\mathrm{GH}}(X,Y) := \inf \left\{\epsilon > 0 \,\left|
\begin{aligned}
&\text{there exist $\epsilon$-approximations } \\
&X \to Y \text{and } Y \to X
\end{aligned} \right. \right\}.
\]
A bijective map $f : X \to Y$ is called an $\epsilon$-{\it almost isometry} if both $f$ and $f^{-1}$ are Lipschitz with Lipschitz constants at most $1+\epsilon$.

\subsection{Alexandrov spaces and good coverings}
We briefly recall the definition of Alexandrov spaces and their properties. 
For details, we refer to \cite{BGP}.
A complete metric space $X$ is called an {\it Alexandrov space} if it is a length space and for any $p \in X$, there exist $\kappa\in\mathbb R$ and a neighborhood $U$ of $p$ such that for any $x,y,z \in U \setminus \{p\}$, we have 
\[
\tilde \angle_\kappa xpy + \tilde \angle_\kappa ypz + \tilde \angle_\kappa zpx \le 2 \pi,
\]
where $\tilde \angle_\kappa xpy$ is defined as the angle of a comparison triangle
$\tilde\Delta xpy=\Delta \tilde x\tilde p\tilde y$ at $\tilde p$ in the complete simply connected surface $M_{\kappa}$ of constant curvature $\kappa$.
It is known that the Hausdorff dimension of $X$ coincides with its Lebesgue covering dimension (\cite{BGP}, \cite{Plaut}), 
which is called the dimension of $X$.
%When $\kappa$ is taken independently on the choice of points $p$, 
When $\kappa$ is chosen to be independent of the choice of points $p \in X$,
we say that $X$ is of curvature $\ge \kappa$. 
%
%
%If $X$ is an Alexandrov space of curvature $\ge \kappa$, for some $\kappa < 0$, the rescaled space $\frac{1}{\sqrt{-\kappa}} X$ is of curvature $\ge -1$. 
%
When $X$ is of dimension $n$, its volume is measured by the $n$-dimensional Hausdorff measure. 

%
%{ a positive constant to its distance function, is of curvature bounded from below by $-1$.}
%{ So, we are going to only deal with the moduli space $\ca A(n,D,v_0)$ consisting of Alexandrov spaces with curvature $\ge -1$. 
%Along with that, all results are stated for $\mathcal A(n, D, v_0)$.}

Complete Riemannian manifolds and orbifolds, the quotient spaces of complete Riemannian manifolds by isometric actions, and the Gromov-Hausdorff limits of sequences of complete Riemannian manifolds with a uniform lower sectional curvature bound, are typical examples of Alexandrov spaces. 

For $m \in \mathbb N$ and $\delta > 0$, a point $p$ in an Alexandrov space $X$ of curvature $\ge \kappa$ 
is called $(m,\delta)$-{\it strained} if there exist pairs of points $\{(a_i, b_i)\}_{i=1}^m$ such that 
\begin{align*}
\hspace{1cm}& \tilde \angle_{\kappa} a_i p b_i > \pi - \delta, && \tilde \angle_{\kappa} a_i p b_j > \pi/2-\delta, 
\hspace{2cm}\\
& \tilde \angle_{\kappa} a_i p a_j > \pi /2 - \delta, && \tilde \angle_{\kappa} b_i p b_j > \pi/2- \delta
\end{align*}
for every $1 \le i \neq j \le m$.
%Here, $\tilde \angle$ means $\tilde \angle_{-1}$. 
The set $\{(a_i, b_i)\}$ is called an $(m, \delta)$-{\it strainer at} $p$. 
The {\it length} $\ell$ of the strainer $\{(a_i, b_i)\}$ at $p$ is defined as %the value: 
\[
\ell := \min \{|p, a_i|, |p, b_i| \mid 1 \le i \le m \}.
\]

From now on, we shall use the convention 
\[
\tilde \angle xyz :=\tilde \angle_{\kappa} xyz,
\]
when the curvature lower bound is understood.

In an $n$-dimensional Alexandrov space $X$, a point $p \in X$ is called $\delta$-{\it regular} if it is $(n, \delta)$-strained and $\delta \ll 1/n$. 
The set of all $\delta$-regular points is called a $\delta$-{\it regular part}, and is denoted by $\ca R_X(\delta)$.

In the present paper, we are concerned with a moduli space of Alexandrov spaces with curvature bounded from below by a uniform constant, say $\kappa$. 
Rescaling the metric, we assume $\kappa=-1$ without loss of generality.
Thus we deal with the moduli space $\ca A(n,D,v_0)$ as explained in the introduction.

\begin{thm}[\cite{BGP}] \label{thm:alm-isom}
Suppose that $X$ is $n$-dimensional and $\delta$ is sufficiently small with $\delta \ll 1/n$.
If $p$ is $(n,\delta)$-strained by an $(n,\delta)$-strainer $\{(a_i, b_i)\}_{i=1}^n$
with length $\ell$, then the map $\varphi$ 
defined by 
\[
\varphi (x) := (|a_1, x|, \ldots, |a_n, x|)
\]
is a $\tau(\delta, \sigma/\ell)$-almost isometry from $B(p, \sigma)$ to an open subset of $\mathbb R^n$.
\end{thm}

%Here we use the symbol $\tau(\delta_1,\delta_2)$ to denote a positive 
%function depending on $\delta_1,\delta_2$ with $\lim_{\delta_1,\delta_2\to 0}\tau(\delta_1,\delta_2) =0$.
We will use Theorem \ref{thm:alm-isom} in Section \ref{sec:gluing}.

Perelman proved the following theorem, called the topological stability theorem.

\begin{thm}[\cite{Per:Alex2}, see also \cite{Kap:stab}] \label{thm:stability}
Let $D > 0$ and $n \in \mathbb N$ be fixed. 
Let $M_j$ be a sequence of $n$-dimensional compact Alexandrov spaces of diameter $\le D$ and curvature $\ge -1$ which converges to an $n$-dimensional compact Alexandrov space $M$ as $j \to \infty$. 
Then, there is $j_0$ such that $M_j$ and $M$ are homeomorphic for all $j \ge j_0$. 

In particular, the set of homeomorphism types of spaces in the moduli space $\mathcal A(n,D,v_0)$ is finite.
\end{thm}

The last statement follows from the fact that $\mathcal A(n,D,v_0)$ is compact 
with respect to the Gromov-Hausdorff distance.

We shall define a new notion of good coverings for {\it metric spaces}. 
In \cite{MY:good}, we have proved that any Alexandrov space has a covering with geometrically and topologically good properties:
%Namely, we have 
\begin{thm}[\cite{MY:good}] \label{thm:good}
%Any finite-dimensional Alexandrov space
%admit a good covering in the sense of Definition \ref{def:good cover}.
%Further, 
For any open covering of a finite-dimensional 
Alexandrov space $X$, there is an open covering $\mathcal U$ of $X$ which is a refinement of the original covering, satisfying the following 
Let $V=\bigcap_{i=0}^k U_{j_i}$ be any non-empty intersection of finitely many elements of $\ca U$.
Then 
\begin{enumerate}
\item $V$ is convex in the sense that every minimal geodesic joining 
any two points in $V$ is contained in $V$; 
\item there exists a point $p \in V$ such that $(V,p)$ is homeomorphic to 
a cone $(C,v)$, where $v$ is the apex of $C$; 
\item there exists a Lipschitz strong deformation retraction $h : V \times [0,1] \to V$ of $V$ to the point $p$ as above (2) such that $|h_t(x), p|$ is non-increasing in 
$t\in [0,1]$ for each $x \in V$. 
\end{enumerate}
\end{thm}
%Note that the Lipschitz homotopy $h$ given in Part (3) of Theorem \ref{thm:good} {provides} a Lipschitz strong deformation retraction to $p$. 

By extracting some fundamental properties that the covering $\mathcal U$
in Theorem \ref{thm:good} posses, we define good coverings for general metric spaces as follows.
\begin{defn} \upshape \label{def:good cover}
A locally finite open covering $\mathcal U=\{U_j\}_{j \in J}$ of a metric space $X$ is {\it good} if it satisfies the following:
\begin{enumerate}
\item the closure of each element of $\mathcal U$ is compact;
\item every non-empty intersection $\bigcap_{i=0}^k U_{j_i}$ of 
finitely many elements of $\ca U$ has a Lipschitz strong deformation retraction to a point $p$.
\end{enumerate} 
%The point $p$ is called the {\it center} of $\bigcap_{i=1}^k U_{j_i}$.
Any such a point $p$ as in (2) is called a {\it center} of $\bigcap_{i=0}^k U_{j_i}$.
We also say that $\mathcal U$ is a good $r$-cover if $\diam(U_j)<r$ for any $j\in J$.
\end{defn}

%%%%%%%%%%%%%%%%%%%%%%%%%%%%%%%%%%%%%%%%%%%%%%%%%%%%%%%%%%%%%%%%%%%%%%%%%%%%%%%%

%\section{Good covering and Lipschitz homotopy types} 
\section{Proof of Theorem \ref{thm:lip-nerve} (compact case)} \label{sec:Lip-homotopy}
\mbox{}

In this section, we prove Theorem \ref{thm:lip-nerve} in the case when $M$ is compact.
We deal with the noncompact case in the next section.
For the proof of Theorem \ref{thm:lip-nerve}, we employ 
a basic strategy in the proof of Theorem 9.4.15 of \cite{Sha}, where 
it is proved that if a topological space has a locally finite covering all of whose 
non-empty intersection are contractible, then it has the same homotopy type as the 
nerve of the covering.
Since the argument there is only topological, we have to proceed 
in the category of (locally) Lipschitz maps. 

%Our argument provides a {locally} Lipschitz version of Theorem 9.4.15 of \cite{Sha}.
%\begin{thm} \label{thm:nerve}
% Let $M$ be a metric space having a good cover $\mathcal U=\{U_j\}_{i\in J}$. 
%Then $M$ has the same {locally} Lipschitz homotopy type as the nerve $\mathcal N_{\mathcal U}$
%of $\mathcal U$.
%\end{thm}

\subsection*{Setting and strategy}

Let $M$ be a compact metric space having a good cover $\mathcal U=\{U_j\}_{j \in J}$. 
Note that $J$ is a finite set since $\ca U$ is locally finite and $M$ is compact.
Let $J=\{1, 2, \ldots, N\}$. 
Let $\ca N_{\ca U}$ denote the {\it nerve} of $\mathcal U$, which is a simplicial complex
%Let $\mathcal U$ be a family of subsets of a set $X$. 
%The $\mathcal N_{\mathcal U}$ of $\mathcal U$ is an abstract complex 
with the set of vertexes $\{U \in \mathcal U \mid U \neq \emptyset \}$, and 
whose $k$-simplices are unordered $(k+1)$-tuples 
$\left< U_{j_0}, U_{j_1}, \dots, U_{j_k} \right>$ of elements in 
$\mathcal U$ so that $\bigcap_{i=0}^k U_{j_i} \neq \emptyset$.
%The vertex $\left< U_{j}\right>$ is denoted by $v_j$.
%
%{ 
%In the situation of actually using this concept, it is assumed that every element of $\mathcal U$ is non-empty and every $U \neq U'$ 
%For a finite abstract complex $\mathcal N$, 
We denote by 
$|\mathcal N_{\ca U}|\subset \mathbb R^N$ its {\it geometric realization}, where we assume that 
$j$-th vertex $v_j :=\left< U_{j}\right>$ of $\mathcal N_{\ca U}$ is given by 
\[
v_j= (0, \ldots, {\overset {j}1}, \ldots, 0)\in \mathbb R^N.
\]

%
%
%
%Here $|\mathcal N|$ is a simplicial complex in the Euclidean space $$, where $m$ is the number of vertices of $\ca N_{\ca U}$, which is defined as follows.
Let $\theta : V(\ca N_{\ca U}) \to [0,1]$ be a function defined on the set of vertices of $\ca N_{\ca U}$ satisfying 
\begin{enumerate}
\item $\sum_{j\in J} \theta(v_j)=1;$
\item $\supp (\theta)$ defines a simplex 
$\sigma_{\theta}$ of $\ca N_{\ca U}$.
\end{enumerate}

%Let $|\ca N_{\ca U}|$ be a geometric realization of $\ca U$.}
Since $\theta$ defines the point $\sum_{j \in J} \theta(v_j) v_j$ of $\sigma_{\theta}$, 
it can be considered as an element of $|\ca N_{\ca U}|$.
%Thus $|\ca N_{\ca U}|$ is defined as the set of such $\theta$'s.
% with the identification $\theta=\sum_{j\in J } \theta(v_j) v_j$, 
From now on, we identify a function $\theta$ satisfying (1), (2) with an element $\sum_{j \in J} \theta(v_j) v_j \in |\ca N_{\ca U}|$. 
That is, 
\[
|\ca N_{\ca U}| = \left\{\theta\equiv\sum_{j\in J} \theta(v_j) v_j \,\middle|\, 
\theta \,\, {\rm satisfies \,\, above}\,\, (1), (2) \right\}. 
\]
This will be useful later on (see the proof of Lemma \ref{lem:section}, for instance).
For any subset $A \subset J$, we put 
\[
U_A:=\bigcap_{j\in A} U_j.
\]

%
%Furthermore, we identify $J$ and the set $V(\ca N_{\ca U})$ of vertices with $\{1, 2, \dots, N\}$, where $\sharp J = N$. 
%}
%{We may assume that, for each distinct pair $j, j' \in J$, we have $U_j \neq U_{j'}$ and for each $j \in J$, $U_j$ is not the empty set.}

Each simplex $\sigma\in \ca N_{\ca U}$ defines a subset $A(\sigma)\subset J$, and we also 
use the symbol 
\[
U_{\sigma}=U_{A(\sigma)}.
\]
By definition, there is a Lipschitz contraction $\varphi:U_{\sigma}\times [0,1]\to U_{\sigma}$
to a point $p_{\sigma}$ of $U_{\sigma}$. 

%Let $L_{\sigma}>0$ denote
%the minimum of those $t_0>0$ such that $\varphi(x, t)=p_{\sigma}$ for all $x\in U_{\sigma}$ and $t\ge t_0$.

%\begin{thm} \label{thm:nerve-control}
% Let $M$ be an Alexandrov space with curvature bounded below
%in $\mathcal A(n,D,v_0)$, 
%and $\mathcal U=\{U_j\}_{i\in J}$ a good cover of $M$. 
%Then $M$ has the same Lipschitz homotopy type as the nerve 
%$\mathcal N_{\mathcal U}$ of $\mathcal U$.
%\end{thm}

%We rescale the size of each simplexes of $\ca N$.

%where 
%%%%$\delta=\delta(\ca U)$ and 
%$N=\# J$.

We define a function $f_j$ on $M$ by 
\[
f_j(x) = \frac{|x,U_j^c|}{|x,U_j^c| + |x,p_j|},
\]
where $U_j^c$ denotes the complement of $U_j$ and $p_j$ is a center of 
$U_j$ to which $U_j$ has a Lipschitz strong deformation retraction.
Since $|x,U_j^c| + |x,p_j|\ge |p_j, U_j^c|/2>0$, 
it is straightforward to check that $f_j$ is Lipschitz.
%with $\supp (f_j)= \bar U_j$.
Set 
\[
\xi_j(x)=\frac{f_j(x)}{\sum_i f_i(x)}.
\]
Then 
$\{\xi_j\}_{j\in J}$ defines a partition of unity dominated with $\ca U$ satisfying %such that 
\begin{enumerate}
\item $\supp (\xi_j) =\bar U_j;$ 
\item each $\xi_j$ is Lipschitz;
\item $\sum_j \xi_j = 1$.
\end{enumerate}

%We denote by $\delta_j={\rm inrad}(U_j)$ the inradius of $U_j$, and set
%%
%\[
% \delta=\delta(\ca U):=\min\{\delta_j\,|\,j\in J\}.
%\]
The polyhedron $|\ca N_{\ca U}|$ has the distance induced from the 
metric of $\mathbb R^{N}$ defined as 
\[
d(x, y) = \max_{1\le i\le N} \, |x_i-y_i|.
\]

In the rest of this section, we are going to construct metric spaces $\ca D(\ca U)$ and $\ca M(p)$ together with natural bi-Lipschitz embeddings 
\beq \label{eq:diagram}
\begin{CD}
 M @>{\tau} >> \ca D(\ca U) \\
 @.
 @VV{\iota}V \\
|\ca N_{\ca U}| @>>{\Psi}> \ca M(p)
\end{CD}
\eeq
%\[
%\xymatrix{
%M \ar[r] & \ca D(\ca U) \ar[d] \\
%|\ca N_{\ca U}| \ar[r] & \ca M(p)
%}
%\]
and prove that their images are Lipschitz strong deformation retracts of the target spaces. 
This strategy comes from \cite{Sha}.
The most complicated part is a construction of a Lipschitz strong deformation retraction from $\ca M(p)$ to $\iota(\ca D(\ca U))$, which will be done simplex-wisely by means of provided Lipschitz strong deformation retractions of $U_\sigma$'s to their centers.
\pmed

We divide the proof into three steps.
\pmed\n

\subsection*{{Step 1}} \,\,
We consider the following subspace of the product metric space 
$|\ca N_{\ca U}|\times M$ defined as 
\[
\ca D(\ca U):=\{(\theta, x)\in |\ca N_{\ca U}|\times M \,|\, x\in U_{\supp\,\theta} \} .
\]
Let $p:\ca D(\ca U)\to |\ca N_{\ca U}|$ and $q:\ca D(\ca U)\to M$ 
be the projection: 
\[
p(\theta, x)=\theta,\,\,\, q(\theta, x)=x.
\]

\psmall
\begin{lem} \label{lem:section}
There exists a Lipschitz map $\tau:M\to \ca D(\ca U)$ such that 
\begin{enumerate}
\item $\tau$ is a section of the map $q$ (i.e., $q \circ \tau =1_M$);
\item $\tau(M)$ is a Lipschitz strong deformation retract of $\ca D(\ca U)$.
%%% \item the length of each Lipschitz deformation in $(2)$ is less than $C\delta$.
\end{enumerate}
\end{lem}

\begin{proof}
Define $\tau:M\to \ca D(\ca U)$ by 
\[
\tau(x) = (\Theta(x), x),
\]
where $\Theta(x)\in |\ca N_{\ca U}|$ is defined by 
$\Theta(x)(v_j) = \xi_j(x)$, $j\in J$.
Obviously %the multiplicity of $\ca U$ is uniformly bounded, 
$\tau$ and $\Theta:M\to |\mathcal N_{\ca U}|$ are Lipschitz.
For any $x\in M$, let $s(x):=\{j\in J\,|\, x\in U_j\}$, which forms a 
simplex of $\ca N_{\ca U}$. 
For any $(\theta, x)\in\ca D(\ca U)$, $\supp (\theta)$ defines a face of 
$s(x)$. Thus we can define the Lipschitz map 
$H:\ca D(\ca U)\times [0,1] \to \ca D(\ca U)$ by 
\beq
H( \theta,x,s) =(s\Theta(x) + (1-s)\theta, x) \label{eq:tau}
\eeq
satisfying $H(\theta,x, 0)=1_{\ca D(\ca U)}(\theta, x)$, 
$H(\theta, x, 1)=(\Theta(x),x)=\tau(x)$ and $H(\tau(x),s)=\tau(x)$
for every $s\in [0,1]$. 
Obviously
$H$ is Lipschitz.
This completes the proof.
\end{proof}
\psmall
\begin{cor} \label{cor:M-D}
$M$ has the same Lipschitz homotopy type as $\ca D(\ca U)$. 
%%%such that the length of each Lipschitz homotopy is less than $C\delta$.
\end{cor}

\begin{proof}
Let $q':\tau(M)\to M$ be the restriction of $q$ to $\tau(M)$.
Since $\tau$ is Lipschitz and $q'$ is $1$-Lipschitz with 
$q' \circ \tau=1_M$ and $\tau \circ q'=1_{\tau(M)}$,
$M$ and $\tau(M)$ are bi-Lipschitz homeomorphic to each other.
The conclusion follows from Lemma \ref{lem:section}.
\end{proof}

\pmed\n
\subsection*{{Step 2.}}\,\,

For $L>0$ (see {\eqref{eq:ell}} for the proper choice of $L$), 
consider the mapping cylinder of $p$:
\[
\ca M(p):= \ca D(\ca U)\times [0,L]\amalg |\ca N_{\ca U}|/(\theta, x,L)\sim \theta.
\]
Recall that 
\[
\ca D(\ca U)=\bigcup_{\sigma\in \ca N_{\ca U}} \sigma\times U_{{\sigma}}
\subset |\ca N_{\ca U}|\times M.
\]
{The canonical correspondence $[(\theta, x,t)]\to (\theta, [(x,t)])$ gives rise to the identification} 
\[
\ca M(p)=\bigcup_{\sigma\in \ca N_{\ca U}} \sigma\times 
K(U_{{\sigma}})
\subset |\ca N_{\ca U}|\times K(M),
\]
where $K(V)=V\times [0,L]/V\times L$ denotes the Euclidean cone. From now on, 
we consider the metric of $\ca M(p)$ induced from that 
of the product metric $|\ca N_{\ca U}|\times K(M)$, where the metric of 
the Euclidean cone 
$K(M)=M\times [0,L]/M\times L$ is defined as 
\begin{align*}
|[x,t], & [x',t']|^2 \\
&= (L- t)^2 + (L-t')^2 -2(L-t)(L-t')\cos (\min\{\pi, |x,x'|\}),
\end{align*}
for $[x,t],[x',t']\in K(M)$.

Note that there is a natural isometric embedding
${\Psi}: |\ca N_{\ca U}| \to \ca M(p)$ defined by 
\[
{\Psi}(\theta)=(\theta, [x,L]) =[\theta]{=(\theta, v_M)},
\]
{where $v_M$ denotes the vertex of $K(M)$.}
\psmall
\begin{lem} \label{lem:N-M}
$|\ca N_{\ca U}|$ is a Lipschitz strong deformation retract of $\ca M(p)$. 
%%%such that the length of each Lipschitz deformation is less than $L$.
\end{lem}
\begin{proof}
Define ${\Psi'}:\ca M(p)\to |\ca N_{\ca U}|$ 
by
\[
{\Psi'}(\theta, [x,t]) = \theta.
\]
Since
\begin{align*}
| {\Psi'}(\theta, [x,t]) , & {\Psi'}(\theta', [x',t']) | \\
& = |\theta,\theta'| \\
& \le \sqrt{|\theta,\theta'|^2 + |[x,t], [x',t']|^2} \\
&=|(\theta, [x,t]), (\theta', [x',t'])|, 
\end{align*} 
and since $|{\Psi}(\theta_1), {\Psi}(\theta_2)| =|\theta_1,\theta_2|$,
both ${\Psi}$ and ${\Psi'}$ are $1$-Lipschitz.

Note that ${\Psi\circ\Psi'}(\theta, [x,t])=[\theta]=(\theta, [x, L])$ and
${\Psi'\circ\Psi}=1_{|\ca N_{\ca U}|}$.
Define $F:\ca M(p)\times [0,1]\to \ca M(p)$ by 
\psmall
\[
F(\theta,[x,t],s) = (\theta, [x, (1-s)t+sL]).
\]
\psmall
Then $F_0=1_{\ca M(p)}$ and $F_1={\Psi\circ\Psi'}$.
We show that $F$ is Lipschitz. Since it suffices to prove that 
it is locally Lipschitz., let us assume that $(\theta,[x,t],s)$ and $(\theta',[x',t'],s')$ are 
close to each other.
We then have
\[
|F(\theta,[x,t],s), F(\theta',[x',t'],s')|^2 = |\theta,\theta'|^2
+ |[x,u], [x',u']|^2,
\]
where we set $u=(1-s)t+sL$, $u'=(1-s')t'+s'L$, and 
\begin{align*}
|[x,u], & [x',u']|^2 \\
&= (L-u)^2 +(L-u')^2 -2(L-u)(L-u')\cos |x,x'| \\
& \le (u-u')^2 + (L-u)(L-u')|x,x'|^2.
\end{align*}
Since $|u-u'|\le (1-s')|t-t'|+(L-t)|s-s'|$, we have 
\begin{align*}
|F(\theta,[x,t], &s), F(\theta',[x',t'],s')|^2 \\
& \le |\theta,\theta'|^2 + 2|t-t'|^2 +2|s-s'|^2+ (L-u)(L-u')|x, x'|^2. \label{eq:Fs}
\end{align*}
Similarly we have 
\begin{align*}
|(\theta,[x,t],s), & (\theta',[x',t'],s')|^2\\
&\ge |\theta,\theta'|^2 + |t-t'|^2+ \frac{1}{2}(L-t)(L-t')|x, x'|^2 {+} |s-s'|^2. 
\end{align*}
Combining those inequalities, we conclude that $F$ is Lipschitz. 
%%%and \eqref{eq:Fs} shows that 
%%%the length of each Lipschitz deformation is less than $L$.
\end{proof}

%
%\pmed\n
%{
%{\rm{\bf Step 3.}
%}

\pmed\n
\subsection*{{Step 3.}}\,\,
{
Let us define $\iota:\ca D(\ca U)\to \ca D(\ca U)\times 0\subset \ca M(p)$ 
by $\iota(\theta,x)=(\theta,x,0)$.
In this last step,} we prove
\psmall
\begin{prop} \label{prop:D-M} {
There exists a Lipschitz strong deformation retraction 
$\Phi:\ca M(p)\times [0,1] \to \ca M(p)$ of $\ca M(p)$ to 
$\ca D(\ca U)\times 0$.}
%$\ca D(\ca U)\times 0$ is a Lipschitz strong deformation 
%retract of $\ca M(p)$. 
\end{prop}
\psmall
The compact case of Theorem \ref{thm:lip-nerve} now follows from Corollary \ref{cor:M-D},
Lemma \ref{lem:N-M} and Proposition \ref{prop:D-M}.

Let 
\beq
L > 6. \label{eq:ell}
\eeq
Let $k_0$ denote the dimension of $\ca N$. 
For each $0\le k\le k_0$,
let $\ca N^{(k)}$ denote the $k$-skeleton of $\ca N_{\ca U}$, and 
$\ca D^k := p^{-1}(|\ca N^{(k)}|)$ and $p^k:=p|_{\ca D^k}:\ca D^k \to \ca N^{(k)}$.
Let $\ca M(p^k)$ denote the mapping cone of $p^k$:
\psmall
\[
\ca M(p^k):= \ca D^k\times [0,L]\amalg |\ca N^{(k)}|/(\theta, x,L)\sim \theta.
\]
\psmall
{
As before, we have 
\[
\ca D^k =\bigcup_{\sigma\in\ca N^{(k)}} \sigma\times U_\sigma,
\]
\[
\ca M(p^k)=\bigcup_{\sigma\in \ca N^{(k)}} \sigma\times 
K(U_{\sigma})
\subset |\ca N_{\ca U}|\times K(M).
\] 
}
\psmall
\begin{lem}\label{lem:induct}
For each $k$, 
{
There exists a Lipschitz strong deformation retraction 
$\Phi^k:\ca M(p^k)\times [0,1] \to \ca M(p^k)$ of $\ca M(p^k)$ to 
$\ca D^k\times 0\bigcup \ca M(p^{k-1})$.}
%
%$\ca D^k\times 0\bigcup \ca M(p^{k-1})$ is a 
%Lipschitz strong deformation retract of $\ca M(p^k)$. 
\end{lem}
\psmall
The construction of the Lipschitz strong deformation retract{ion $\Phi^k$} in 
Lemma \ref{lem:induct} will be done simplex-wisely. 
{This is based on the following sublemma.}
\psmall
\begin{sublem} \label{sublem:simplex-wise}
{For each $k$-simplex $\sigma\in\ca N_{\ca U}$,
there exists a Lipschitz strong deformation retraction %$\Phi_{\sigma}$
of $\sigma\times K(U_{\sigma})$ to $(\sigma\times U_{\sigma}\times 0)\bigcup \pa\sigma\times K(U_{\sigma})$.
}
\end{sublem}
\psmall
{
Since $\pa\sigma\times K(U_{\sigma})\subset \ca M(p^{k-1})$, 
applying Sublemma \ref{sublem:simplex-wise} to each $k$-simplex
of $\ca N_{\ca U}$, we obtain Lemma \ref{lem:induct}.

By using Lemma \ref{lem:induct} {repeatedly}, 
we have a finite sequence of Lipschitz retractions:
\psmall
\beq
\begin{split}
 \ca M(p)= \ca M(p^{k_0}) \longrightarrow & \ca D(\ca U)\times 0\bigcup \ca M(p^{k_0-1}) 
 \longrightarrow \cdots \\
 & \longrightarrow D(\ca U)\times 0\bigcup \ca M(p^0) 
 \longrightarrow \ca D(\ca U)\times 0.
 \label{eq:sequence}
\end{split}
\eeq
\psmall\n
From \eqref{eq:sequence}, we conclude that $\ca D(\ca U)\times 0$ is a Lipschitz strong 
deformation retract of $\ca M(p)$. 
Thus all we have to do is to prove Sublemma \ref{sublem:simplex-wise}.
%This will finish the proof of Proposition \ref{prop:D-M}.

\psmall
\begin{rem} \upshape 
From \eqref{eq:sequence}, one might think that
$k_0=\dim \ca N_{\ca U}<\infty$ is essential in the argument below.
However, we can generalize the argument of this section to the general case of
$\dim \ca N_{\ca U}=\infty$.
This will be verified in Section \ref{sec:non-cpt}.
\end{rem}
}

\pmed
The following is the important first step in the proof of Sublemma 
\ref{sublem:simplex-wise}, which is the case of $k=0$.

\begin{claim} \label{claim:key}
Let $U$ be an element of $\ca U$.
Then there exists a Lipschitz strong deformation retraction
{
$K(U) \times [0,1] \to K(U)$
of $K(U)$ to $U\times 0$.}
\end{claim}

\begin{proof}
{Let $\varphi:U \times [0,L] \to U $ be a Lipschitz strong deformation retraction to $p \in U$. 
We may assume that $\varphi(x,t) = p$ for all $t \ge L/2$ and $x \in U$.}
Define the retraction $r:K(U)\to U\times 0\subset K(U)$ by
\[
r([x,t]) := [\varphi(x,t), 0].
\]
First we show that $r$ is Lipschitz. 
Again we way assume that 
$[x,t]$ and $[x',t']$ are sufficiently close.
Note that
\begin{align*}
|r([x,t]), r([x',t'])| &\le |[\varphi(x,t),0],[\varphi(x',t),0]| +
|[\varphi(x',t),0],[\varphi(x',t'),0]| \\
&\le L|\varphi(x,t), \varphi(x',t)| +
L|\varphi(x',t), \varphi(x',t')| \\
&\le CL|x,x'| + CL|t-t'|.
\end{align*}
%Here and hereafter 
{From here on,}
we use the symbols $C, C_1,C_2,\ldots$ to 
denote some uniform positive constants.

If both $t$ and $t'$ are greater than $L/2$, then 
$\varphi(x,t)=\varphi(x',t')=p$. Therefore we may assume that 
$t,t'\le L/2$. Then we have 
\begin{align*}
|[x,t], [x',t']|^2 &\ge (t-t')^2 + (L-t)(L-t')|x,x'|^2/2 \\
&\ge (t-t')^2 + (L^2/8)|x,x'|^2.
\end{align*}
Combining the two inequalities, we have
\begin{align*}
|r([x,t]), r([x',t'])| &\le CL|x,x'| + C_1L|t-t'| \\
& \le C (1+L)|[x,t], [x',t']|
\end{align*}
Now let $g:[0,L]\times [0,1]\to [0,1]$ be a Lipschitz function such that 
\begin{itemize}
\item $g(t,s) =1$ on $[0,L]\times [0,1/3];$
\item $g(t,1)=0$ for all $0\le t\le L$,
\end{itemize}
and define $\Phi:K(U)\times [0,1]\to K(U)$ by
\[
\Phi([x,t],s) = [\varphi(x,st), g(t,s)t]
\]
Note that $\Phi([x,t],0)=[x,t]$, $\Phi([x,t],1)=r([x,t])$.

To show that $\Phi$ is Lipschitz, let $([x,t],s)$ and $([x',t'],s')$ be
elements of $K(U)\times [0,1]$ sufficiently close to each other.
By triangle inequalities, it suffices to show the following:
\begin{enumerate}
\item $|\Phi([x,t],s), \Phi([x',t],s)| \le C_1|[x,t],[x',t]|;$
\item $|\Phi([x,t],s), \Phi([x,t'],s)| \le C_2L|[x,t],[x,t']|;$
\item $|\Phi([x,t],s), \Phi([x,t],s')| \le C_3L(1+L)|s-s'|$.
\end{enumerate}

We show $(1)$:
\begin{align*}
|\Phi([x,t],s), \Phi([x',t],s)| &=|[\varphi(x,st), g(t,s)t],[\varphi(x',st), g(t,s)t]| \\
& \le |L-g(t,s)t||\varphi(x,st), \varphi(x',st)| \\
& \le |L-g(t,s)t| C |x,x'| \\
& \le |L-g(t,s)t| |x, x'|.
\end{align*}
If $s\le 1/3$, then
$|L-g(t,s)t| |x,x'| = (L-t)|x,x'| \le 2|[x,t],[x',t]|$.
If $s\ge 1/3$ and $t\ge L/2$, then $ts\ge L_0$, and therefore
$|\Phi([x,t],s), \Phi([x',t],s)|=0$.
If $s\ge 1/3$ and $t\le L/2$, then
$|L-g(t,s)t| |x,x'|\le L|x,x'|\le 2(L-t)|x,x'| \le 3|[x,t],[x',t]|$.

We show $(2)$:
\begin{align*}
|\Phi([x,t],s), \Phi([x,t'],s)|^2 &=|[\varphi(x,st), g(t,s)t],[\varphi(x,st'), g(t',s)t']|^2 \\
& \le |g(t,s)t-g(t',s)t'|^2 + L^2 |\varphi(x,st), \varphi(x,st')|^2, \\
& \le |g(t,s)t-g(t',s)t'|^2 + L^2 C_1|st-st'|^2,
\end{align*}
where obviously
\begin{align*}
|g(t,s)t-g(t',s)t'| & \le |g(t,s)t-g(t',s)t| + |g(t',s)t-g(t',s)t'| \\
& \le C(1+L)|t-t'|.
\end{align*} 
Thus we have $ |\Phi([x,t],s), \Phi([x,t'],s)|\le C_2(1+L)|[x,t],[x,t']|$.

We show $(3)$:
\begin{equation}
\begin{aligned}
|\Phi([x,t],s), \Phi([x,t],s')|^2 &=|[\varphi(x,st), g(t,s)t],[\varphi(x,s't), g(t,s')t]|^2 \\
& \le |g(t,s)t-g(t,s')t|^2 + C_1 L^2 |st-s't|^2 \\
& \le C_2L^2|s-s'|^2 + C_3L^4|s-s'|^2\\
& \le C_4L^2(1+L^2)|s-s'|^2. 
\end{aligned}
\label{eq:Ps}
\end{equation}
This shows that $\Phi$ is Lipschitz, together with \eqref{eq:Ps}
this completes the proof of Claim \ref{claim:key}.
\end{proof}

Next we consider the general case.

\begin{proof}[Proof of Sublemma \ref{sublem:simplex-wise}]
Let $\sigma$ be any simplex of $\ca N$. 
%Since $\diam(\sigma)$, $\delta$ and $L$ are all comparable, 
Note that $\sigma\times 0\cup\pa\sigma\times [0,L]$ is a 
Lipschitz strong deformation retract of $\sigma\times [0,L]$.
Let $r:\sigma\times [0,L]\to \sigma\times 0\cup\pa\sigma\times [0,L]$
be a Lipschitz strong deformation retraction defined by the radial projection from the point
$(x^*,2L)\in\sigma\times\mathbb R$, where $x^*$ is the barycenter of $\sigma$. 
Let us {represent $r$ as}
\[
r(x,t) =(\psi_0(x,t), u(x,t))\in \sigma\times 0\cup\pa\sigma\times [0,L]\subset\sigma\times [0,L].
\]
Define the retraction $f:\sigma\times K(U) \to \sigma\times U\times 0\cup\pa\sigma\times K(U)$
by
\[
f(x,[y,t])=(\psi_0(x,t), [\varphi(y,t-u(x,t)), w(x,t)]),
\]
where $w:\sigma\times [0,L]\to [0,L]$ is defined as follows:
%%%%%%%%%%%
Let us consider the following closed subsets of $\mathbb R^{N+1}$:
\begin{align*}
& \Omega_0=\{(x,t)\in \sigma\times [0,L]\,|\,u(x,t)\le L/10\}, \\
& \Omega_1=\{(x,t)\in \sigma\times [0,L]\,|\,u(x,t)\ge L/2\}.
\end{align*}
Note that $|\Omega_0,\Omega_1|\ge c>0$ for some constant $c>0$.
Let $s_i(x,t)=|(x,t), \Omega_i|$, $i=1,,2$, and define $w$ by 
\[
w(x,t) := \frac{s_1(x,t)}{s_0(x,t)+s_1(x,t)} u(x,t) + \frac{s_0(x,t)}{s_0(x,t)+s_1(x,t)} t.
\]
%%%%%%%%%%%
Note that $w$ is Lipschitz and has the property 

\[
w(x,t) = \begin{cases}
u(x,t)\,\, &{\rm if} \,\,u(x,t)\le L/10 \\
t \,\, &{\rm if} \,\, u(x,t)\ge L/2.
\end{cases}
\]
Note also that $f$ is the identity on $ \sigma\times U\times 0\cup\pa\sigma\times K(U)$,
and therefore it defines a retraction of $ \sigma\times U\times 0\cup\pa\sigma\times K(U)$.
We show that $f$ is Lipschitz. It suffices to show that the second component 
\[
f_2(x,[y,t])=([\varphi(y,t-u(x,t)), w(x,t)])
\]
of $f$ is Lipschitz.
As before, we may assume that $(x, [y,t])$ and $(x',[y',t'])$ are sufficiently close to each other.
Letting $u=u(x,t)$, $u'=u(x',t)$, $w=w(x,t)$, $w'=w(x',t)$ we have 
\begin{align*}
|f_2(x, &[y,t]), f_2(x',[y,t])|^2 \\
&=|[\varphi(y,t-u), w], [\varphi(y,t-u'),w']|^2 \\
& \le (L-w)^2 + (L-w')^2 -2(L-w)(L-w')\cos |\varphi(y,t-u), \varphi(y,t-u')| \\
& \le (w-w')^2 +(L-w)(L-w')|\varphi(y,t-u), \varphi(y,t-u')|^2 \\
& \le (w-w')^2 + C_1L^2(u-u')^2\\
& \le C_2(1+L^2)|x,x'|^2,
\end{align*}
and
\begin{align*}
|f_2(x,[y,t]), f_2(x,[y',t])| &=|[\varphi(y,t-u), w], [\varphi(y',t-u),w]| \\
& \le (L-w)|\varphi(y,t-u), \varphi(y',t-u)|,
\end{align*}
where since $|[y,t],[y',t]|\ge (1/2)(L-t)|y,y'|$, we may assume that $t\ge 9L/10$.
If $t\ge 9L/10$ and $u(x,t)\le L/2$, then $\varphi(\cdot, t-u)=p$.
If $t\ge 9L/10$ and $u(x,t) > L/2$, then $w(x,t)=t$, and we have 
\begin{align*}
|f_2(x,[y,t]), f_2(x',[y,t])| &\le (L-t) C|y,y'| \\
& \le C|[y,t], [y',t]|.
\end{align*}
Finally letting $u=u(x,t)$, $u'=u(x,t')$, $w=w(x,t)$, $w'=w(x,t')$ we have 
\begin{align*}
|f_2(x,[y,t]), f_2(x,[y,t'])|^2 &=|[\varphi(y,t-u), w], [\varphi(y,t-u'),w']|^2 \\
& \le (w-w')^2 + L^2|\varphi(y,t-u), \varphi(y,t-u')|^2 \\
& \le (w-w')^2 + C_1L^2(u-u')^2 \\
& \le C_2(1+L^2)|t-t'|^2.
\end{align*}
Thus $f$ is Lipschitz.

Now define the homotopy $\Phi:\sigma\times K(U)\times[0,1]\to \sigma\times K(U)$ by
\begin{align*}
&\Phi(x, [y,t],s) \\
& = ((1-s)x+s\psi_0(x,t), [\varphi(y,\mu(s)(t-u(x,t)),(1-\nu(s))t+\nu(s)w(x,t)]),
\end{align*}
where $\mu$ and $\nu$ are Lipschitz functions on $[0,1]$ satisfying
\[
\mu(s) = \begin{cases}
1 \,\, & {\rm if} \,\,s \ge 2/3 \\
0 \,\, & {\rm if} \,\, s\le 1/2,
\end{cases}
\,\, \nu(s) = \begin{cases}
1 \,\, &{\rm if} \,\,s \ge 3/4 \\
0 \,\, &{\rm if} \,\, s\le2/3,
\end{cases}
\]
Obviously, $\Phi(\cdot, 0)=1_{\sigma\times K(U)}$, 
$\Phi(\cdot, 1)=f$ and $\Phi(\cdot,s)$ fixes each point of 
$\sigma\times U\times 0 \cup\pa\sigma\times K(U)$.
We show that $\Phi$ is Lipschitz. 
It suffices to show that the second component 
\[
\Phi_2(x,[y,t],s)=([\varphi(y,\mu(s)(t-u(x,t)),(1-\nu(s))t+\nu(s)u(x,t)])
\]
of $\Phi$ is Lipschitz.
As before, we may assume that $(x, [y,t],s)$ and $(x',[y',t'],s')$ are sufficiently close to each other.
Letting $u=u(x,t)$, $u'=u(x',t)$, $w=w(x,t)$, $w'=w(x',t)$ and $\mu=\mu(s)$, $\nu=\nu(s)$, we have 
\begin{align*}
|\Phi_2(x, & [y,t],s), \Phi_2(x',[y,t],s)|^2 \\
&=|[\varphi(y,\mu(t-u)), (1-\nu)t +\nu w], [\varphi(y,\mu(t-u')), (1-\nu)t +\nu w']|^2 \\
& \le \nu^2(w-w')^2 +L^2|\varphi(y,\mu(t-u)), \varphi(y,\mu(t-u')|^2 \\
& \le\nu^2(w-w')^2 + C_1L^2\mu^2(u-u')^2\\
& \le C_2(1+L^2)|x,x'|^2,
\end{align*}
and
\begin{align*}
|\Phi_2(x, & [y,t],s), \Phi_2(x,[y',t],s)| \\
&=|[\varphi(y,\mu(t-u)), (1-\nu)t +\nu w], [\varphi(y',\mu(t-u)), (1-\nu)t +\nu w]| \\
& \le (L-(1-\nu)t -\nu w)|\varphi(y,\mu(t-u), \varphi(y',\mu(t-u)| \\
& \le (L-(1-\nu)t )C|y,y'|, 
\end{align*}
where, if $t<9L/10$, then $L|y,y'|\le CL|[y,t], [y',t]|$. Hence 
we may assume that $t\ge 9L/10$.
If $u(x,t)>L/2$ then $w(x,t)=t$. If $u(x,t)\le L/2$ and $s\ge 2/3$, then $\mu(s)=1$ and $\varphi(\cdot, \mu(t-u))=p$.
If $u(x,t)\le L/2$ and $s\le 2/3$, then $\nu=0$. Thus we conclude that 
\begin{align*}
|\Phi_2(x,[y,t],s), \Phi_2(x,[y',t],s)| &\le (L-t) C|y,y'| \\
& \le C|[y,t], [y',t]|.
\end{align*}
Next letting $u=u(x,t)$, $u'=u(x,t')$, $w=w(x,t)$, $w'=w(x,t')$, we have 
\begin{align*}
|\Phi_2(x, & [y,t],s),\Phi_2(x,[y,t'],s)|^2 \\
&=|[\varphi(y,\mu(t-u)), (1-\nu)t +\nu w], [\varphi(y,\mu(t'-u')), (1-\nu)t' +\nu w']|^2 \\
& \le ((1-\nu)(t-t') + \nu(w-w'))^2 + L^2 |\varphi(y,\mu(t-u), \varphi(y, \mu(t'-u'))| \\
& \le C(1+L^2)(t-t')^2.
\end{align*}
Finally letting $\mu'=\mu(s')$, $\nu'=\nu(s')$, we have 
\begin{align*}
|\Phi_2(x, &[y,t],s), \Phi_2(x,[y,t],s')|^2 \\
&=|[\varphi(y,\mu(t-u)), (1-\nu)t +\nu w], [\varphi(y,\mu'(t-u)), (1-\nu')t +\nu' w]|^2 \\
& \le (t(\nu'-\nu)+w(\nu-\nu'))^2 + L^2 |\varphi(y,\mu(t-u), \varphi(y, \mu'(t-u))|^2 \\
& \le L^2(\nu-\nu')^2 + C_1L^2(\mu-\mu')^2 \\
& \le C_2L^2(s-s')^2.
\end{align*}
Thus $\Phi$ is Lipschitz.
%%% such that 
%%%the length of each Lipschitz deformation is less than $CL$.
This completes the proof of Sublemma \ref{sublem:simplex-wise}.
\end{proof}

This completes the proof of Proposition \ref{prop:D-M}.
We have just proved the compact case of
Theorem \ref{thm:lip-nerve}.

{
By the above discussion, we have the following commutative diagram:
\begin{equation*} \label{eq:diagram2}
\begin{CD}
 M @>{\tau} >> \ca D(\ca U) \\
 @V{\Theta}VV
 @VV{\iota}V \\
|\ca N_{\ca U}| @<<{\Psi'} < \ca M(p)
\end{CD}
\end{equation*}

From Lemmas \ref{lem:section}, \ref{lem:N-M} and Proposition \ref{prop:D-M}
together with \eqref{eq:diagram}, we have the following.
}
%explicit Lipschitz homotopy
%equivalence $\Theta:M\to |\ca N_{\ca U}|$.}
%
%By repeating Lemma \ref{lem:induct} $k_0$-times, we see that 
%$\ca M(p)$ is a Lipschitz strong deformation retract of 
%$\ca D(\ca U)\times 0\cup\ca M(p^0)$.
%%% such that 
%%%the length of each Lipschitz deformation is less than $CL$.
%It follows from Claim \ref{claim:key} that 
%$\ca M(p)$ is a Lipschitz strong deformation retract of 
%$\ca D(\ca U)$. 
%%%%such that 
%%%%the length of each Lipschitz deformation is less than $CL(1+L)$.
%Thus we conclude that $M$ has the same Lipschitz homotopy
%type as $\ca N$.
%%% such that 
%%%the length of each Lipschitz deformation is less than $CL(1+L)$. 

{
\begin{cor} \label{cor:homo-inv1}
Let $M, \mathcal U, \{\xi_j\}_{j \in J}$ be the same as in this section. 
Then {the} natural map
\[
\Theta : M \ni x \mapsto (\xi_j(x))_{j \in J} \in |\ca N_{\ca U}|
\]
is a Lipschitz homotopy equivalence. 
\end{cor}
}
\psmall
 {
\begin{cor} \label{cor:homo-inv2}
Let $M$, $\mathcal U=\{U_j\}_{j=1}^N$ and $\ca N_{\ca U}$ be the same as in this section,
and $\zeta:|\ca N_{\ca U}|\to M$ a Lipschitz homotopy inverse to 
$\Theta : M \to |\ca N_{\ca U}|$. For every $\theta\in |\ca N_{\ca U}|$,
let $\sigma$ be the open simplex of $\ca N_{\ca U}$ containing $\theta$
with $\sigma=\langle U_{j_0}, \ldots, U_{j_k}\rangle$.
Then we have 
\beq
 \zeta(\theta) \in \bigcup_{i=0}^k \,U_{j_i}. \label{eq:zeta}
\eeq
\end{cor}
\begin{proof}
Let $H:\ca D(\ca U)\times [0,1] \to \ca D(\ca U)$ 
be a Lipschitz strong deformation retraction of $\ca D(\ca U)$
to $\tau(M)$ given in \eqref{eq:tau},
and set $H_1:=H(\,\cdot\,, 1)$.
Let $\Phi:\ca M(p)\times [0,1] \to \ca M(p)$ be a 
Lipschitz strong deformation retraction of $\ca M(p)$ to 
$\ca D(\ca U)\times 0$ given in Proposition \ref{prop:D-M}, 
and set $\Phi_1:=\Phi(\,\cdot\,, 1)$.
From our argument in this section, we have the following commutative 
diagram:
\begin{equation*} \label{eq:diagram2}
\begin{CD}
 M @<{\tau^{-1}\circ H_1} << \ca D(\ca U) \\
 @A{\zeta}AA
 @AA{\Phi_1}A \\
|\ca N_{\ca U}| @>>{\Psi} > \ca M(p)
\end{CD}
\end{equation*}
Note that $\tau^{-1}\circ H_1(\mu,x)=x$ for every $(\mu,x)\in\ca D(\ca U)$.
Therefore, %from $\Psi(\theta)=(\theta,v_M)$, 
we can write 
\[
 \Phi_1\circ \Psi(\theta)=(\eta(\theta),\zeta(\theta))\in \ca D(\ca U),
\]
where $\eta(\theta)\in |\ca N_{\ca U}|$ and $\zeta(\theta)\in M$.
If $\sigma_1$ denotes the open simplex of $\ca N_{\ca U}$ containing
$\eta(\theta)$, then it follows from the definition of $\ca D(\ca U)$
that $\zeta(\theta)\in U_{\sigma_1}$.
From Sublemma \ref{sublem:simplex-wise} together with \eqref{eq:sequence},
we see that $\sigma_1$ is a face of $\sigma$, which yields \eqref{eq:zeta}.
\end{proof}

}

\section{Noncompact case} \label{sec:non-cpt}
{We prove Theorem \ref{thm:lip-nerve} for} {the general case}.
Let $M$ be a {$\sigma$-compact} metric space admitting a good covering 
$\mathcal U=\{U_j\}_{j\in J}$. {From the local finiteness of $\mathcal U$, 
$J$ is countable, and therefore we may assume $J=\mathbb N$.}
%$\mathcal U=\{U_j\}_{j\in \mathbb N}$.
%
{
Since $\ca U$ is locally finite, the number of $U_j$'s meeting 
each $\bar U_i$ is finite. It follows that the nerve $\ca N_{\ca U}$ is locally finite.
Note that $|\ca N_{\ca U}|\subset \mathbb R^{\infty}$ in this case.
Note that the Lipschitz constant 
of the strong deformation retraction $U_j\times [0,1]\to U_j\times [0,1]$ of $U_j$ 
to a point of $U_j$ depends on $j$, and that 
$\dim \ca N_{\ca U}=\infty$ in general.
From the local finiteness of $\ca N_{\ca U}$, basically we can do the same 
construction as in 
Section \ref{sec:Lip-homotopy} to obtain the spaces $\ca D(\ca U)$, $\ca M(p)$
in the general case, too.
We also have the natural embeddings in a similar manner: 
\begin{equation*}
\begin{CD}
 M @>{\tau} >> \ca D(\ca U) \\
 @.
 @VV{\iota}V \\
|\ca N_{\ca U}| @>>{\Psi}> \ca M(p)
\end{CD}
\end{equation*}
Note that the map $\tau:M\to \ca D(\ca U)$ defined by 
\[
\tau(x) = (\Theta(x), x),\quad \Theta(x)(v_j) = \xi_j(x) \quad (j\in J)
\]
is locally bi-Lipschitz.
In a way similar to Corollary \ref{cor:M-D}, we see that 
$\tau(M)$ has the same locally Lipschitz homotopy type as $\ca D(\ca U)$. 
Note also that 
the natural embedding
${\Psi}: |\ca N_{\ca U}| \to \ca M(p)$ defined by 
\[
{\Psi(\theta)=(\theta, [x,L]) =[\theta]}
\]
is isometric, and we see that 
$|\ca N_{\ca U}|$ is a locally Lipschitz strong deformation retract of $\ca M(p)$
in a way similar to Lemma \ref{lem:N-M}.
Therefore to complete the proof of Theorem \ref{thm:lip-nerve} in the general case,
we only have to check the following:

\begin{lem} 
There exists a locally Lipschitz strong deformation retraction of
$\ca M(p)$ to $\ca D(\ca U)\times 0$.
%and a locally Lipschitz homotopy 
%$\Phi:\ca M(p)\times [0,1] \to \ca M(p)$
%for the retraction $r$. 
\end{lem}

\begin{proof}
Recall that in the compact case in Section \ref{sec:Lip-homotopy}, the strong deformation retraction 
$\Phi:\ca M(p)\times [0,1] \to \ca M(p)$ of $\ca M(p)$
to $\ca D(\ca U)\times 0$ is constructed simplex-wisely 
from higher dimensions to lower dimensions through Lemma \ref{lem:induct}.
Therefore the Lipschitz constant of $\Phi$ depends on
 {the dimension $k_0$ of $\ca N_{\ca U}$.}%

Since $\dim \ca N_{\ca U}$ could be infinite in the present case, 
first of all, we have to {verify} that the map 
$\Phi:\ca M(p)\times [0,1] \to \ca M(p)$ is well-defined.
For any point $(\theta, [x,t])\in \ca M(p)$, let $\sigma_\theta$
be the simplex whose interior contains $\theta$.
Let $L$ be the star of $\sigma_\theta$, which is a finite subcomplex
because of the local finiteness of $\ca N_{\ca U}$.
Then 
\[
\ca M_L(p) := \bigcup_{\sigma\in L} \, \sigma\times K(U_{\sigma})
\]
provides an open subset of $\ca M(p)$ containing $(\theta, [x,t])$.
Set 
\[
\ca D_L:= \bigcup_{\sigma\in L} \, \sigma\times U_{\sigma}.
\]
From the argument in Lemma \ref{lem:induct}, we have 
a Lipschitz strong deformation retraction 
$\Phi_L:\ca M_L(p) \times [0,1]\to \ca M_L(p)$
of $\ca M_L(p)$ to $ \ca D_L\times 0$.
Since $\ca M(p)$ is a locally finite union of such open subsets
$\ca M_L(p)$, we can construct a strong deformation retraction 
$\Phi:\ca M(p) \times [0,1]\to \ca M(p)$
of $\ca M(p)$ to $ \ca D\times 0$
simplex-wisely in a similar manner.
Since $\Phi|_{\ca M_L(p)\times [0,1]}=\Phi_L$ and the construction is simplex-wise,
$\Phi$ is locally Lipschitz. 
This completes the proof. 
\end{proof}
}

\section{Proof{s} of Theorem \ref{thm:lip-conv} and Corollary \ref{cor:Lip-conti}} \label{sec:appl}

{
%Let us give a proof of Theorem \ref{thm:lip-conv}. 
To prove Theorem \ref{thm:lip-conv}, we need the following result, 
which follows from the proof of {\cite[Theorem 1.2]{MY:good}}. 
 
\begin{thm} [\cite{MY:good}] \label{thm:approximation}
{For every $M\in\ca A(n,D,v_0)$, let $M_i$ be a sequence in $\ca A(n,D,v_0)$ 
converging  to $M$ as $i \to \infty$.  
%Let $\theta_i : M \to M_i$ be an $\epsilon_i$-approximation with $\lim_{i \to \infty} \epsilon_i = 0$. 
Then for any $\mu>0$}, there exists a good $\mu$-covering $\ca U = \{U_j\}_{j \in J}$ of $M$ satisfying the following. 
For every $\epsilon_i$-approximations $\phi_i : M \to M_i$ with {$\epsilon_i \to 0$},  
there exist a good $2\mu$-covering $\ca U_i = \{U_{ij}\}_{j \in J}$ of $M_i$ and $\nu_i$-approximations $\varphi_i : M \to M_i$ with $\nu_i\to 0$ 
such that for sufficiently large $i$
\begin{enumerate}
\item $\varphi_i(p_j)$ is a center of $U_{ij}$ in the sense of Definition {\ref{def:good cover}} for every $j \in J$; 
\item {the} corresponding $U_{j} \mapsto U_{ij}$ induces an isomorphism $\ca N_{\ca U} \to \ca N_{\ca U_i}$ between the nerves of $\ca U$ and $\ca U_i$; 
\item $\lim_{i \to \infty} \sup_{x \in M} |\phi_i(x), \varphi_i(x)| = 0$.
\end{enumerate}
%In particular, $\varphi_i$ is a $\mu_i$-approximation for some $\mu_i \to 0$.
\end{thm}
\psmall\n
\begin{defn} \upshape
{We call such a $\ca U_i$ {given in Theorem \ref{thm:approximation} } a {\it lift} of $\ca U$ with respect to $\varphi_i$}. 
\end{defn}

} 

%Let us construct a good covering $\ca U = \{U_j\}_{j \in J}$ of $M$ described in \cite{MY:good}.
%Let $\{p_j \}_{j \in J} \subset M$ be chosen so that each $p_j$ is a center of $U_j$ in the sense of Definition \ref{def:good cover}.
%Then, for sufficiently large $i$, there exists $2 \epsilon_i$-approximation $\tilde \theta_i : M \to M_i$ and a good covering $\ca U_i = \{U_{ij}\}_{j \in J}$ of $M_i$ with the same index set $J$ as $\ca U$ such that 
%for every $j$, $\tilde \theta_i(p_j)$ is a center of $U_{ij}$. 
%An important point is that $\lim_{i \to \infty}|\tilde \theta_i(p_j), \theta_i(p_j)| = 0$ holds for every $j \in J$. 

%First we give a proof of Theorem \ref{thm:lip-conv}. 
%{Now, we give a proof of Theorem \ref{thm:lip-conv}.}
%{Before proving Theorem \ref{thm:lip-conv}, we recall a result in \cite{MY:good}. 
%\begin{theorem}[\cite{MY:good}]
%Let $M, M'$ be Alexandrov spaces belonging to a fixed moduli space $\ca A(n,D,v_0)$. 
%In \cite{MY:good}, it was proved that there exists a $\mu > 0$ depending on 
%\end{theorem}
%}
\begin{proof}[Proof of Theorem \ref{thm:lip-conv}]
%Let $\ca A = \ca A(n,D,v_0)$. 
Due to \cite[Theorem 1.2]{MY:good}, there exist $\epsilon>0$ and finitely many spaces $M_1, \dots, M_N \in {\ca A(n,D,v_0)}$ and finite simplicial complexes $K_1, \dots, K_N$ such that 
\psmall
\begin{itemize}
\item $\bigcup_{i=1}^N U_{\epsilon}^{\mathrm{GH}}(M_i) = \ca A(n,D,v_0)$; 
\item any $M \in U_\epsilon^{\mathrm{GH}}(M_i)$ admit a good covering whose nerve complex is isomorphic to $K_i$.
\end{itemize}
\psmall\n
Here, $U_\epsilon^{\mathrm{GH}}(X)$ denotes the $\epsilon$-neighborhood of $X$ in $\ca A(n,D,v_0)$ 
with respect to the Gromov-Hausdorff distance. 
From this and Theorem \ref{thm:lip-nerve}, we obtain the first conclusion of Theorem \ref{thm:lip-conv}. 

We prove the second conclusion by contradiction.
Suppose it does not hold.
Then we would have sequences $\{M_i\}$, $\{M_i'\}$ in $\ca A(n,D,v_0)$ with 
$d_{GH}(M_i,M_i')<{\delta_i}$, $\lim \delta_i=0$, together with 
$\delta_i$-approximation $\theta_i:M_i\to M_i'$ such that 
\begin{align}
\sup_{x\in M_i} {|\theta_i(x), {k_i}(x)|} >c>0, \label{eq:contra}
\end{align}
for any Lipschitz 
homotopy equivalence ${k_i}:M_i\to M_i'$, where $c$ is a constant not depending 
on $i$.
Passing to a subsequence, we may assume that 
both $M_i$ and $M_i'$ converge to an Alexandrov space $M \in \ca A(n,D,v_0)$.
We introduce a new positive number $\mu \ll c$.
By {Theorem \ref{thm:approximation}, one can take 
a good $\mu$-cover $\ca U$} ${= \{U_j\}_{j \in J}}$ of $M$ {
and a good $2\mu$-cover $\mathcal U_i$ {$=\{U_{ij}\}_{j \in J}$} of $M_i$ 
such that $\ca U_i$ is a lift of $\ca U$ with respect to some $\nu_i$-approximation
$\varphi_i:M\to M_i$, where $\lim_{i\to\infty}\nu_i=0$.
%
%the nerve {$\ca N_{\ca U_i}$} is isomorphic to the nerve {$\ca N_{\ca U}$ for sufficiently large $i$}, {where $J$ is a finite set}.
%{Here, the isomorphism $\alpha_i : \ca N_{\ca U} \to \ca N_{\ca U_i}$ is given by 
%{the} correspondence $U_j \mapsto U_{ij}$.}
%%
%Furthermore, there {exists a $\mu_i$}-approximation $\varphi_i : M \to M_i$ 
%with $\lim_{i \to \infty} \mu_i = 0$ such that
%$\ca U_i$ is a lift of $\ca U$ with respect to $\varphi$.
%
% for each fixed center $p_j$ of $U_j \in \ca U$, its image $p_{{ij}} := \varphi_i(p_j)$ becomes a center of {$U_{ij}$.} 
%, where $\ca U_i$ has a canonical isomorphism $\alpha_i : \ca N_{\ca U} \to \ca N_{\ca U_i}$ given by $U_i \mapsto U_{ij}'$.
%We call {such a} $\ca U_i$ a lift of $\ca U$ with respect to $\varphi_i$. 
%{In particular, for $x \in M$, if $U_j \in \ca U$ contains $x$, then }
Let $\psi_i :M_i \to M$ be an $\nu_i$-approximation, which is {an} almost inverse of $\varphi_i$, in the sense that 
\[
   \sup_{x \in M_i} |\varphi_i \circ \psi_i(x), x| \le \nu_i, \quad
   \sup_{x \in M} |\psi_i \circ \varphi_i(x), x| \le \nu_i.
\]
%For $M_i'$, one can also obtain a good $2\mu$-cover $\ca U_i'$ which is a lift of $\ca U$ with respect to a suitable $\mu$-approximation $\varphi_i' : M \to M_i'$. 
{Note that $\theta_i \circ \varphi_i:M\to M_i'$ is {a} $2( \delta_i + \nu_i)$-approximation.
%Hence, using 
Applying Theorem \ref{thm:approximation} to $\theta_i \circ \varphi_i$, 
%the stability result mentioned in the previous paragraph, 
we also obtain a $\nu_i'$-approximation $\varphi_i' : M \to M_i'$ with 
$\lim_{i\to\infty}\nu_i'=0$ and a lift $\ca U_i'$ ${= \{U_{ij}' \}_{j \in J}}$ of $\ca U$ with respect to $\varphi_i'$
{such that
\beq  \label{eq:theta-varphi}
\lim_{i\to\infty} \sup_{x\in M} |\theta_i \circ\varphi_i(x), \varphi_i'(x)| = 0     
\eeq
Set  $p_{ij}:= \varphi_i(p_j)$ and $p_{ij}':= \varphi_i'(p_j)$, which are  
centers of  $U_{ij}$ and $U_{ij}'$ respectively}.
It follows that %for every $j \in J$
\beq  \label{eq:1.2-01}
   %{ |\theta_i (p_{ij}), p_{ij}'| \le  \mu,} \quad
\sup_{x \in M_i}|\theta_i(x), \varphi_i' \circ \psi_i(x)| \le \mu
\eeq
\noindent
for large $i$.
{Let  $\alpha_i : \ca N_{\ca U} \to \ca N_{\ca U_i}$ and 
$\alpha_i' : \ca N_{\ca U} \to \ca N_{\ca U_i'}$ be the
isomorphisms given by the correspondence $U_j \mapsto U_{ij}$
and $U_j \mapsto U_{ij}'$ respectively}.
%In particular, $|\theta_i (p_{ij}), p_{ij}'| \le 3 \mu$ for.
We now consider the following diagram: 
\[
\xymatrix{
M_i \ar[d]_{\Theta_i} \ar@<0.2em>[r]^{\psi_i} \ar@<-0.2em>@{.>}[r]_{h_i} 
& M \ar@<0.2em>[d]^\Theta \ar@<0.2em>[r]^{\varphi_i'} \ar@<-0.2em>@{.>}[r]_{g_i} & M_i' \\ 
%&&\\
|\ca N_{\ca U_i}| 
%\ar@/_1.6em/[rrdd]_{U_{ij} \mapsto U_j} 
\ar[r]_{\alpha_i^{-1}} 
& |\ca N_{\ca U}| 
\ar@<0.2em>[u]^{\zeta}
%\ar@/^0.6em/[ull]^{U_j \mapsto U_{ij}} 
%\ar@/_0.6em/[uurr]_{U_j \mapsto U_{ij}'} 
\ar[r]_{\alpha_i'} 
& 
|\ca N_{\ca U_i'}| \ar[u]_{\zeta_i'} 
}
\]
Here, $\Theta, \zeta, \Theta_i, \zeta_i'$ are  maps given by Corollaries \ref{cor:homo-inv1} and \ref{cor:homo-inv2}, for $(M, \ca U)$, $(M_i, \ca U_i)$ and $(M_i', \ca U_i')$
respectively. 
For instance, $\Theta (x) = (\xi_j(x))_{j \in J}$ for $x \in M$, where $(\xi_j)_{j \in J}$ is a partition of unity by Lipschitz functions subordinate to $\left\{\bar U_j \right\}_{j \in J}$, and $\zeta$ is a Lipschitz homotopy inverse of $\Theta$ given by Corollary \ref{cor:homo-inv2}.
Now, we consider the compositions
\[ {
h_i:= \zeta\circ \alpha_i^{-1}\circ \Theta_i, \quad
g_i:= \zeta_i'\circ \alpha'_i\circ \Theta,}
\]
%    h_i$ (and $g_i$, respectively) of $\Theta_i$, $\alpha_i^{-1}$ and $\zeta$ (and $\Theta$, $\alpha_i'$ and $\zeta_i'$, respectively). 
which are Lipschitz homotopy equivalences
satisfying
%Applying Theorem \ref{thm:lip-nerve}, {Corollaries \ref{cor:homo-inv1} and \ref{cor:homo-inv2}}, 
%Therefore, we have %locally 
%{Lipschitz} homotopy equivalences $h_i:M_i \to M$ and $g_i:M\to M_i'$ which are {\bf $\mu$-close} to $\psi_i$ and $\varphi_i'$, respectively, in the sense that 
\begin{equation} \label{eq:10mu}
\sup_{x \in M_i} |\psi_i(x), h_i(x)| \le 10 \mu, \hspace{1em} \hspace{1em} \sup_{x \in M} |\varphi_i'(x), g_i(x)| \le 10 \mu. 
\end{equation}
Indeed, for $x \in M_i$, $\Theta_i(x)$ is contained in a unique open simplex $\left< U_{i j_0}, \cdots, U_{i j_k} \right> \in \ca N_{\ca U_i}$. 
Then, $\alpha_i^{-1} \circ \Theta_i(x)$ is contained in $\left< U_{j_0}, \dots, U_{j_k} \right>$.
By the property of $\zeta$ stated in Corollary \ref{cor:homo-inv2}, we have 
$h_i(x) \in {U_{j_{\ell}}}$ for some $0\le \ell \le k$.
On the other hands, since $x \in U_{i j_{\ell}}$, %for any $0 \le \ell \le k$, 
we have $|x, p_{i {j_\ell}}| \le 2\mu$ and $|\psi_i(x), p_{j_\ell}| \le 3 \mu$.
Therefore, we obtain 
\begin{align*}
|h_i(x), \psi_i(x)| &\le |h_i(x), p_{j_\ell}| + |p_{j_\ell}, \psi_i(x)| \le 4 \mu.
\end{align*}
Thus we obtain \eqref{eq:10mu} for $h_i$. 
Similarly we obtain \eqref{eq:10mu} for $g_i$.
% close to $\varphi_i$ and $\varphi_i'$
%Then, from the triangle inequality, the centers $p_{{ij}}$ and $p_{{ij}}'$ 
%of elements of $\ca U_i$ and $\ca U'_i$ 
%satisfy $|p_{{ij}}, \theta_i(p_{{ij}}')| \le 3 \mu$.
%
%Applying Theorem \ref{thm:lip-nerve}, we have 
%Lipschitz homotopy equivalences $h_i:M_i \to M$ 
%and $g_i':M\to M_i'$.
%We may assume that {centers} $p_{{ij}}$ and $p_{{ij}}'$
%of {elements} of $\ca U_i$ and $\ca U'_i$ satisfy
%$|p_{{ij}}, \theta_i(p_{{ij}}')|\ll \mu$. 
It follows {from \eqref{eq:10mu} and \eqref{eq:1.2-01}} that 
$ \sup_{x\in M_i} |\theta_i(x), g_i \circ h_i(x)| \le 100 \mu$,
}}
which is a contradiction to \eqref{eq:contra}
\end{proof}
\psmall
For two metric spaces $A$ and $B$, let us denote by
\[
[A,B]_{{\mathrm{loc}\text{-}\mathrm{Lip}}} 
\]
the set of all {locally} Lipschitz homotopy classes of locally Lipschitz maps from $A$ to $B$.
Let us denote by 
\[
[A,B]
\]
the set of all homotopy classes of continuous maps from $A$ to $B$.
For another metric space $C$ and a locally Lipschitz map $f : A \to B$, we define a map $f^\ast : [B,C]_{{\mathrm{loc}\text{-}\mathrm{Lip}}} \to [A,C]_{{\mathrm{loc}\text{-}\mathrm{Lip}}}$ (and $f^\ast : [B,C] \to [A,C]$) by $f^\ast (g) := g \circ f$ up to locally Lipschitz homotopy (and up to homotopy, respectively). 
From the definition, for a locally Lipschitz map $g : B \to C$, we have 
\begin{equation} \label{eq:contravariant}
(g \circ f)^\ast = f^\ast \circ g^\ast.
\end{equation}

\begin{proof}[Proof of Corollary \ref{cor:Lip-conti}]
Let us fix a good cover $\mathcal U$ of { a $\sigma$-compact metric space $X$, and let $K$ be 
%$\mathcal N := \mathcal N_\mathcal{U}$ 
the geometric realization of the nerve} 
of $\mathcal U$. 
From \cite[Corollary 1.3]{MY:LLC}, the induced map
\[
[{K}, Y]_{{\mathrm{loc}\text{-}\mathrm{Lip}}} \to [{K}, Y]
\]
is bijective. 
By Theorem \ref{thm:lip-nerve}, ${K}$ is locally Lipschitz homotopy equivalent to $X$.
Let $f : X \to {K}$ and $g : {K} \to X$ {be} locally Lipschitz homotopy equivalences such that 
{$g \circ f$ and $f \circ g$ are 
locally Lipschitz homotopy eqiuvalent to $\mathrm{id}_X$ and $\mathrm{id}_K$, respectively.}
%$\mathrm{id}_X$ and ${f \circ g} \sim_{{\mathrm{loc}\text{-}\mathrm{Lip}}} \mathrm{id}_{{K}}$.
By the contravariant property \eqref{eq:contravariant}, the induced maps $g^\ast : [X,Y]_{{\mathrm{loc}\text{-}\mathrm{Lip}}} \to [{K}, Y]_{{\mathrm{loc}\text{-}\mathrm{Lip}}}$ and $f^\ast : [{K}, Y]_{{\mathrm{loc}\text{-}\mathrm{Lip}}} \to [X, Y]_{{\mathrm{loc}\text{-}\mathrm{Lip}}}$ are mutually inverse.
So are $f^\ast : [{K}, Y] \to [X, Y]$ and $g^\ast : [X, Y] \to [{K}, Y]$.
%%% satisfy that $f^\ast \circ g^\ast$ and $g^\ast \circ f^\ast$ are identities. 
These imply the conclusion.
\end{proof}

{Remark that in the statement of Corollary \ref{cor:Lip-conti}, if $X$ is compact, we obtain a natural bijection between the set of all Lipschitz homotopy classes of Lipschitz maps from $X$ to $Y$ and $[X,Y]$.}

A refinement of Corollary \ref{cor:Lip-conti} is the following: 
\begin{cor} \label{cor:001}
Let $X$ and $Y$ be as in Corollary \ref{cor:Lip-conti}. 
For any continuous function $\epsilon : Y \to (0,\infty)$ and any continuous map $f : X \to Y$, there is a locally Lipschitz map $g : X \to Y$ which is homotopic to $f$ and satisfies 
\[
|f(x), g(x)| < \epsilon(f(x))
\]
for every $x \in X$.
\end{cor}
\begin{proof}
%This is proved in a manner similar to the proof of Corollary \ref{cor:Lip-conti}, by using \cite[Corollary 4.4]{MY:LLC}. 
%%%the fact that every continuous map $h$ from a simplicial complex $\mathcal N$ to $Y$ is homotopic to a locally Lipschitz map $k : \mathcal N \to Y$ such that $d(h(x), k(x)) < \epsilon(h(x))$ for every $x \in \mathcal N$. 
%By using \cite[Corollary 4.4]{MY:LLC} and discussing in a similar manner to the proof of Corollary \ref{cor:Lip-conti}, we can prove Corollary \ref{cor:001}.
{This follows from \cite[Corollary 4.4]{MY:LLC} and a discussion similar to the proof of Corollary \ref{cor:001}.}
\end{proof}

%%%%%%%%%%%%%%%%%%%%%%%%%%% Oct 04 %%%%%%%%%%%%%%%%%%%%%%%%%%%%%%%%%%%%%%%%

\section{Gluing with an almost isometry} \label{sec:gluing}

For a small $1/n\gg \delta>0$, let $\mathcal R_M(\delta)$ the open set of $M$ consisting of all 
$(n,\delta)$-{strained} points, which is called the $\delta$-regular part of $M$.
In this section we prove Theorem \ref{thm:gluing} by making use of the notion of 
center of mass developed in \cite{BGP} (see \cite{GK} for the original idea). 

\begin{proof} [Proof of Theorem \ref{thm:gluing}]
Let $\theta:M\to M'$ be an $\e$-approximation.
Take $\mu>0$ such that the closed $3\mu$-neighborhood of $D$ is contained in 
$\mathcal R_M(\delta)$. 
Let $D_1$ be the closed $2\mu$-neighborhood of $D$.
We also denote by $D_0$ the closed $\mu$-neighborhood of $D$. 
By \cite{BGP} and \cite{Ya:conv}, 
for small enough $\e>0$ with $\e\ll \mu$, we have a $\tau(\delta)$-almost isometric map 
%$g:D_1\to g(D_1)\subset \mathcal R_{M'}(2\delta)$ 
\[
 g:D_1\to g(D_1)\subset \mathcal R_{M'}(2\delta) 
\]
such that $d(g(x),\theta(x))<\tau(\e)$ for all $x\in D_1$.
On the other hand from Theorem \ref{thm:lip-conv}, we have a 
Lipschitz homotopy equivalence 
\[
 f:M\to M'
\]
%$f:M\to M'$ 
such that $d(f(x),\theta(x))<\tau(\e)$ for all $x\in M$.

We shall construct a Lipschitz homotopy equivalence $h:M\to M'$
such that $h=g$ on $D$ and $h=f$ on $M\setminus D_0$.
Denote by $E$ the closure of $D_1\setminus D$. Take $R>0$
such that each point $x\in E$ has an $(n,\delta)$-strainer of length $> R$.
Let $\{x_i\}_{i=1}^N\ \subset E$ be a maximal family with $|x_i,x_j|\ge \delta R/2$ for each $i\neq j$.
Then $\{B_i\}_{i=1}^N$ with $B_i:= B(x_i,\delta R/2)$ gives a covering of $E$.
By Theorem \ref{thm:alm-isom}, for each $1\le i\le N$ there are $\tau(\delta)$-almost isometric maps 
\[
 f_i:B(x_i, 2\delta R) \to \mathbb R^n, \quad f_i':B(g(x_i), 2\delta R) \to \mathbb R^n.
\]
%Letting $B_i' = B(g(x_i),\delta R)$. Similarly we have 
Let 
\[
 d(x) =\min \{|D,x|, \mu \}.
\]
Note that the multiplicity of the covering $B(x_i, 2\delta R)$ is uniformly bounded 
by {a constant} $C_n$.

For every $x\in B_i$, let 
\[
 h_i^0(x) := (f_i')^{-1}\left( \frac{d(x)}{\mu} f_i'(f(x)) +\left(1-\frac{d(x)}{\mu}\right)f_i'(g(x))\right).
\]
This extends to a Lipschitz map $h_i: \overline{M\setminus E}\cup B_i \to M'$ satisfying
%\psmall
\begin{align*}
 h_i(x) = \begin{cases}
 g(x), \,\,\, &x\in D \\
 f(x), \,\,\, &x\in \overline{M\setminus D_0},
 \end{cases}
\end{align*}
%\psmall
and $|\theta(x), h_i(x)|<\tau(\e)$ for all $x\in \overline{M\setminus E}\cup B_i$.

Now we are going to glue these Lipschitz maps $\{h_i\}$ to get a Lipschitz map $h:M\to M'$.
Define a Lipschitz cut-off function $\varphi_i:M\to \mathbb R$ by
\psmall
\begin{align*}
 \varphi_i(x) := \begin{cases}
 1- \frac{|x,x_i|}{\delta R}, & \,\,\, x\in B(x_i, \delta R) \\
 0, &\,\,\, {\rm otherwise}.
 \end{cases}
\end{align*}
\psmall\n
Let $F_i := \overline{M\setminus E}\cup B_1\cup\cdots \cup B_i $,
and set 
\[
 \psi_i(x) :=\sum_{j=1}^{i}\varphi_j(x),\,\, x\in M.
\]
Assuming that 
$h_{1\cdots i}: F_i\to M'$ is already defined in such a way that
\psmall
\begin{equation}
 \begin{cases}
 & |\theta(x), h_{1\cdots i}(x)|<\tau(\e), \, x\in F_i\\
 &h_{1\cdots i}(x) = \begin{cases}
          g(x), \,\,\, &x\in D \\
          f(x), \,\,\, &x\in \overline{M\setminus D_0}, \label{eq:h1i}
          \end{cases}
\end{cases}
%\begin{equation}
% \begin{aligned}
% & |\theta(x), h_{1\cdots i}(x)|<\tau(\e), \, x\in F_i\\
% &h_{1\cdots i}(x) = \begin{cases}
%          g(x), \,\,\, &x\in D \\
%          f(x), \,\,\, &x\in \overline{M\setminus D_0}, \label{eq:h1i}
%          \end{cases}
%\end{aligned}
\end{equation}
\psmall\n
define $h_{1\cdots i+1}: F_{i+1}\to M'$ by
\psmall
\begin{align*}
 h_{1\cdots i+1}(x) := \begin{cases}
 h_{1\cdots i}(x), \hspace{3.5cm} x\in F_i\setminus B_{i+1}, \\
 (f_{i+1}')^{-1} \left( \left(1-\frac{\varphi_{i+1}(x)}{\psi_{i+1}(x)}\right) f_{i+1}'(h_{1\cdots i}(x)\right) \\
 \hspace{1cm} + \frac{\varphi_{i+1}(x)}{\psi_{i+1}(x)} f_{i+1}'(h_{i+1}(x))), 
 \hspace{0.8cm}\,\,\,
 x\in B_{i+1}. 
 \end{cases}
\end{align*}
\psmall\n
Note that $ h_{1\cdots i+1}$ also satisfies \eqref{eq:h1i}.

%$|h_{1\cdots i}(x), \theta(x)|<\tau(\e)$ for all $ x\in F_{i+1}$, and 
%\begin{align*}
% h_{1\cdots i}(x) = \begin{cases}
% g(x), \,\,\, &x\in D \\
% f(x), \,\,\, &x\in \overline{M\setminus D_1}.
% \end{cases}
%\end{align*}
Finally we set $h:=h_{1\cdots N}:M\to M'$. Note that $|h(x), \theta(x)|<\tau(\e)$ for all $x\in M$, and 
\psmall
\begin{align*}
 h(x) = \begin{cases}
 g(x), \,\,\, &x\in D \\
 f(x), \,\,\, &x\in \overline{M\setminus D_0}.
 \end{cases}
\end{align*}
\psmall
Similarly we define $h':M'\to M$ by using the Lipschitz homotopy inverse $f'$ of $f$,
 $g':=g^{-1}$, $D':=g(D)$, $D_0':=g(D_0)$, $D_1':=g(D_1)$d and $d'=d\circ g^{-1}$
 in place of $f$, $g$, $D$, $D_1$ and $d$.
Note that every $y\in D_1'$ has $(n,2\delta)$-strainer of length $>R/2$.
Obviously, $|h'\circ h(x), x|<\tau(\e)$ and 
\psmall
\begin{align*}
 h'\circ h(x) = \begin{cases}
 x, \,\,\, &x\in D \\
 f'\circ f (x), \,\,\, &x\in \overline{M\setminus D_0}.
 \end{cases}
\end{align*}
\psmall
{To construct a Lipschitz homotopy between $1_M$ and $h'\circ h$,
 we use a method developed in \cite{GP}.
We }consider the product space $M\times M$ and denote by $\Delta\subset M\times M$ the diagonal.
Introduce a positive constant $\sigma$ with $\e\ll \sigma\ll \mu$ and 
take a sequence $0<\sigma_i <\sigma$ with $\lim\sigma_i = 0$. 
For every ${\bf x}:=(x_1,x_2)\in D_1\times D_1\cap A(\Delta;\sigma_i,\sigma)$,
let $y$ denote the midpoint of a minimal geodesic joining $x_1$ and $x_2$,
where $A(\Delta;\sigma_i,\sigma):=\overline{B(\Delta,\sigma)\setminus B(\Delta,\sigma_i)}$ is the annulus.
Note that ${\bf y}:=(y,y)$ is the foot of a minimal geodesic from $\bf x$ to $\Delta$.
It is possible to take points $z_1$ and $z_2$ of $M$ such that 
\begin{align*}
 & \tilde \angle yx_i z_i > \pi -\tau(\delta), i=1,2, \\
 & |y,z_1| =|y, z_2|, \,\, |\Delta, \bf z|=\sigma,
\end{align*}
where ${\bf z}:=(z_1,z_2)$.
Then a direct computation shows that 
\[
 |z_i, y| > |z_i, x_i| + (1-\tau(\delta)) |x_i, y|, \,\, i=1,2,
\]
and
\[
 | {\bf z}, {\bf y}| > |{\bf z}, {\bf x}| + (1-\tau(\delta)) |{\bf x}, {\bf y}|,
\]
which yields that 
\[
 \tilde \angle {\bf z x y} > \pi -\tau(\delta).
\]
The above argument shows that the distance function $d_{\Delta}$ from $\Delta$ is $(1-\tau(\delta))$-regular
on $D_1\times D_1\cap A(\Delta;\sigma_i,\sigma)$.
Now we consider a smooth approximation of a neighborhood $U_i$ 
of $D_1\times D_1\cap A(\Delta;\sigma_i,\sigma)$.
By \cite{KMS} and \cite{MY:3alex}, there are a smooth manifold $N_i$ and a bi-Lipschitz 
homeomorphism $\Phi_i:U_i\to N_i$
together with a gradient like unit vector field $X_i$ for $d_{\Delta}\circ\Phi_i^{-1}$ defined on $N_i$ 
such that if $\phi_i(\Phi_i({\bf x}), t)$ denotes the integral curves of $-X_i$ starting at $\Phi_i({\bf x})$, then 
for each ${\bf x}\in D_1\times D_1\cap A(\Delta;\sigma_i,\sigma)$ with $|{\bf x},\Delta|=\sigma$, 
%for each ${\bf x}$ with $|{\bf x},\Delta|=\sigma$, 
\[
 |\Phi_i^{-1}\circ\phi_i(\Phi_i({\bf x}),t_0),\Delta|=\sigma_i,
\]
for some $t_0<2(\sigma-\sigma_i)$.
By combining the flow curves $\{\Phi_i^{-1}\circ \phi_i(\Phi_i({\bf x}), t)\}_i$, we obtain a Lipschitz
flow $\phi$ on $D_1\times D_1\cap \overline{B}(\Delta, \sigma)$ such that 
for each ${\bf x}\in D_1\times D_1\cap \overline{B}(\Delta, \sigma)$ with $|{\bf x},\Delta|=\sigma$, $\phi({\bf x},s_0)\in \Delta$
%flow $\phi$ on $\overline{B}(\Delta, \sigma)$ such that 
%for each ${\bf x}$ with $|{\bf x},\Delta|=\sigma$, $\phi({\bf x},s_0)\in \Delta$
for some $s_0<2\sigma$.
For ${\bf x}=(x, h'\circ h(x))$, if we denote $\phi({\bf x}, t) =(\phi^1({\bf x}, t), \phi^2({\bf x}, t))$, 
the union of $\phi^1({\bf x}, t)$ and $\phi^2({\bf x}, 1-t)$ provides the desired 
Lipschitz homotopy between $1_M$ and $h'\circ h$ on $D_1$.

We have just constructed a Lipschitz homotopy $H(x,t)$ between $1_M$ and $h'\circ h$ on $D_1$.
Recall that we have a Lipschitz homotopy $F(x,t)$ between $1_M$ and $f'\circ f$.
We have to glue $F$ and $H$ to get a Lipschitz homotopy $G(x,t)$ between $1_M$ and $h'\circ h$ defined on $M$. 
Let $\rho:M\times [0,1]\to [0,1]$ be a Lipschitz function such that
\begin{align*}
 \rho(x,t) = \begin{cases}
 0, \,\,\, & {\rm on} \,\, D\times [0,1]\cup D_0\times [1/2,1], \\
 1, \,\,\, & {\rm on} \,\, \overline{M\setminus D_1}\times [0,1].
 \end{cases}
\end{align*}
\psmall
For every $(x,t)\in B_i\times [0,1]$, let 
\[
 G_i^0(x,t) := f_i^{-1}\left( \rho(x,t) f_i( F(x,t)) + (1-\rho(x,t))f_i(H(x,t))\right).
\]
This extends to a Lipschitz map 
$G_i:\overline{M\setminus E}\cup B_i\times [0,1]\to M$ satisfying
\psmall
\begin{align*}
 G_i(x,t) = \begin{cases}
 x, \,\,\, & {\rm on}\,\,\,\overline{M\setminus E}\cup B_i\times 0,\\
 f'\circ f(x), \,\,\, & {\rm on}\,\,\,\overline{M\setminus E}\cup B_i\times 1, \\
 H(x,t), & {\rm on}\,\,\, D\times [0,1], \\
 F(x,t), \,\,\, & {\rm on}\,\,\, \overline{M\setminus D_1}\times [0,1]. \\
 \end{cases}
\end{align*}
\psmall\n
Assuming that 
$G_{1\cdots i}: F_i\times [0,1]\to M$ is already defined in such a way that
\begin{equation}
% \begin{aligned}
 %& |\theta(x), h_{1\cdots i}(x)|<\tau(\e), \\
 G_{1\cdots i}(x,t) = \begin{cases}
 x, \,\,\, & {\rm on}\,\,\, F_i\times 0, \\
 f'\circ f(x), \,\,\, & {\rm on}\,\,\,F_{i}\times 1, \\
 H(x,t), \,\,\, & {\rm on}\,\,\, D\times [0,1], \\
 F(x,t), \,\,\, & {\rm on} \,\,\, \overline{M\setminus D_1}\times [0,1], 
 \end{cases}
%\end{aligned}
\end{equation}
define $G_{1\cdots i+1}: F_{i+1}\times [0,1]\to M$ by
\begin{align*}
 G_{1\cdots i+1}(x,t) := \begin{cases}
 G_{1\cdots i}(x,t), \hspace{2.7cm} (x,t)\in (F_i\setminus B_{i+1})\times [0,1], \\
 (f_{i+1})^{-1} \left( \left(1-\frac{\varphi_{i+1}(x)}{\psi_{i+1}(x)}\right) f_{i+1}(G_{1\cdots i}(x,t)\right) \\
\hspace{0.5cm} 
+ \frac{\varphi_{i+1}(x)}{\psi_{i+1}(x)} f_{i+1}(G_{i+1}(x,t))), 
 \hspace{0.5cm} \,\,\,(x,t)\in B_{i+1}\times [0,1]. 
 \end{cases}
\end{align*}
\psmall
Finally set $G:=G_{1\cdots N}$. 
Obviously $G$ is Lipschitz, and $G=H$ on $D\times [0,1]$ and 
 $G=F$ on $\overline{M\setminus D_1}\times [0,1]$, and 
thus $G$ is a required Lipschitz homotopy 
between $1_M$ and $h'\circ h$. 
Similarly we obtain a Lipschitz homotopy between $1_{M'}$ 
and $h\circ h'$. 
This completes the proof of Theorem \ref{thm:gluing}.
\end{proof}

\section{Further problems}

It is quite natural to expect that there should exist uniform Lipschitz 
constants of the Lipschitz homotopies in Theorems \ref{thm:lip-conv}
and \ref{thm:gluing}.

%%%%start of the bibliography

\end{document}